\documentclass{compositio}
\usepackage{amssymb}
\usepackage{hyperref}
\usepackage{amsmath, amsthm, graphics,epic,eepic}
\usepackage{xypic}     
\LaTeXdiagrams    
\usepackage[all]{xy}
\xyoption{2cell} \UseAllTwocells \xyoption{frame} \CompileMatrices
\usepackage{latexsym}
\usepackage{epsfig}
\usepackage{amsfonts}
\usepackage{cite}
\usepackage{enumerate}
\usepackage{times}
\usepackage{hyperref}

\allowdisplaybreaks[3]
\newtheorem{prop}{Proposition}
\newtheorem{lem}[prop]{Lemma}
\newtheorem{cor}[prop]{Corollary}
\newtheorem{thm}[prop]{Theorem}
\theoremstyle{definition}
\newtheorem{defn}[prop]{Definition}
\newtheorem{rmk}[prop]{Remark}
\newtheorem{notation}[prop]{Notation}

\newcommand{\bpsi}{\bar{\psi}}

\newcommand{\fb}{\mathbf{f}}

\newcommand{\bSigma}{\mathbf{\Sigma}}

\newcommand{\X}{\mathcal{X}}

\newcommand{\C}{\mathcal{C}}

\newcommand{\F}{\mathcal{F}}

\newcommand{\U}{\mathcal{U}}

\newcommand{\sL}{\mathcal{L}}

\newcommand{\bs}{\mathbf{s}}
\newcommand{\bc}{\mathbf{c}}
\newcommand{\bt}{\mathbf{t}}
\newcommand{\btau}{\boldsymbol{\tau}}
\newcommand{\bq}{\mathbf{q}}
\newcommand{\sH}{\mathcal{H}}
\newcommand{\sO}{\mathcal{O}}
\newcommand{\T}{\mathcal{T}}

\newcommand{\bp}{\mathbf{p}}

\def\<{\left\langle}
\def\>{\right\rangle}

\DeclareMathOperator{\Contr}{Contr}
\DeclareMathOperator{\val}{val}
\DeclareMathOperator{\inv}{inv}
\DeclareMathOperator{\ev}{ev}
\let\Box=\relax
\DeclareMathOperator{\Box}{Box} 
\newcommand{\Boxshape}{\mathchar"403}

\DeclareMathOperator{\Spec}{Spec}
\DeclareMathOperator{\Res}{Res}
\DeclareMathOperator{\res}{res}
\DeclareMathOperator{\lcm}{lcm}
\DeclareMathOperator{\Aut}{Aut}
\DeclareMathOperator{\ord}{ord}
\DeclareMathOperator{\NE}{NE}
\DeclareMathOperator{\Pic}{Pic}
\DeclareMathOperator{\Hom}{Hom}
\DeclareMathOperator{\im}{Im}
\DeclareMathOperator{\Ker}{Ker} 
\DeclareMathOperator{\can}{can}
\DeclareMathOperator{\Frac}{Frac} 
\DeclareMathOperator{\Cone}{Cone} 
\DeclareMathOperator{\Sym}{Sym} 
 
\DeclareMathOperator{\age}{age} 
\DeclareMathOperator{\Lie}{Lie} 
\DeclareMathOperator{\ch}{ch}

\newcommand{\eT}{e_{\TT}}
\newcommand{\cC}{\mathcal{C}}
\newcommand{\mov}{\text{\rm mov}}

\newcommand{\CR}{\text{\rm CR}}

\newcommand{\tor}{\text{\rm tor}}
\newcommand{\tw}{\text{\rm tw}}
\newcommand{\untw}{\text{\rm un}}
\newcommand{\GW}{\text{\rm GW}}
\newcommand{\Cstar}{\mathbb{C}^\times}
\newcommand{\Mbar}{\overline{\mathcal{M}}}
\newcommand{\CC}{\mathbb{C}}
\newcommand{\PP}{\mathbb{P}}
\newcommand{\QQ}{\mathbb{Q}}
\newcommand{\RR}{\mathbb{R}}
\newcommand{\ZZ}{\mathbb{Z}}
\newcommand{\TT}{\mathbb{T}}
\newcommand{\LL}{\mathbb{L}}
\newcommand{\RC}{\text{\rm \bf RC}}

\begin{document}

\title{A Mirror Theorem for Toric Stacks}

\author[Coates]{Tom Coates}
\email{t.coates@imperial.ac.uk}
\address{Department of Mathematics\\ Imperial College London\\ 
180 Queen's Gate\\ London SW7 2AZ\\ United Kingdom}

\author[Corti]{Alessio Corti}
\email{a.corti@imperial.ac.uk}
\address{Department of Mathematics\\ Imperial College London\\ 
180 Queen's Gate\\ London SW7 2AZ\\ United Kingdom}

\author[Iritani]{Hiroshi Iritani}
\email{iritani@math.kyoto-u.ac.jp}
\address{Department of Mathematics\\ Graduate School of Science\\ 
Kyoto University\\ Kitashirakawa-Oiwake-cho\\ Sakyo-ku\\ 
Kyoto\\ 606-8502\\ Japan}

\author[Tseng]{Hsian-Hua Tseng}
\email{hhtseng@math.ohio-state.edu}
\address{Department of Mathematics\\ Ohio State University\\ 
100 Math Tower, 231 West 18th Ave. \\ Columbus \\ OH 43210\\ USA}

\date{\today}



\begin{abstract}
  We prove a Givental-style mirror theorem for toric Deligne--Mumford
  stacks $\X$.  This determines the genus-zero Gromov--Witten
  invariants of $\X$ in terms of an explicit hypergeometric function,
  called the $I$-function, that takes values in the Chen--Ruan
  orbifold cohomology of $\X$.
\end{abstract}

\maketitle

\tableofcontents

\section{Introduction}
In this paper we prove a mirror theorem that determines the genus-zero
Gromov--Witten invariants of smooth toric Deligne--Mumford stacks.
Toric Deligne--Mumford stacks are generalizations of toric varieties
\cite{BCS}, and our mirror theorem generalizes Givental's mirror
theorem for toric manifolds \cite{Givental:toric}.  Following Givental
\cite{gi2}, the genus-zero Gromov--Witten theory of a toric
Deligne--Mumford stack $\X$ can be encoded in a Lagrangian cone
$\sL_\X$ contained in an infinite-dimensional symplectic vector space
$\sH$. Universal properties of Gromov--Witten invariants of $\X$
translate into geometric properties of $\sL_\X$. See
Section~\ref{sec:GW_basics} for an overview.  In
Sections~\ref{sec:toric_mir_thm}--\ref{sec:pf_mir_thm} of this paper
we establish a mirror theorem for a smooth toric Deligne--Mumford
stack $\X$. Roughly speaking our result states that the {\em extended
  $I$-function}, which is a hypergeometric function defined in terms
of the combinatorial data defining $\X$, lies on the Lagrangian cone
$\sL_\X$. The precise statement is Theorem~\ref{I_is_on_the_cone}
below.

Our mirror theorem (Theorem~\ref{I_is_on_the_cone}) has a number of applications.  It has been used to give explicit formulas for genus-zero Gromov--Witten invariants of toric Deligne--Mumford stacks and, when combined with the Quantum Lefschetz theorem \cite{cg,ccit}, to prove a mirror theorem for convex toric complete intersection stacks \cite{ccit2}.  Special cases of Theorem~\ref{I_is_on_the_cone} have been used (as conjectures, proven here) to construct an integral structure on quantum orbifold cohomology of toric Deligne--Mumford stacks, to study Ruan's Crepant Resolution Conjecture, to compute open-closed Gromov--Witten invariants \cite{ir, ChChLaTs, flt}, to prove mirror theorems for open Gromov--Witten invariants \cite{ChChLaTs2}, and to prove mirror theorems for certain toric complete intersection stacks \cite{Iritani}. Theorem~\ref{I_is_on_the_cone} will have further applications in the future: it allows a full proof of the Crepant Resolution Conjecture in the toric case, and a description of the quantum $D$-module of a toric Deligne--Mumford stack. We will discuss these applications elsewhere.  Theorem~\ref{I_is_on_the_cone} extends previous works on the Gromov--Witten theory of certain classes of toric stacks, including weighted projective spaces \cite{AGV2,cclt,Mann,gs,acvg}, one-dimensional toric Deligne--Mumford stacks \cite{Johnson, mt}, toric orbifolds of the form $[\CC^n/G]$ \cite{CC,BC,JPT,Brini--Cavalieri}, and the ambient space for the mirror quintic \cite{sl}.

Since our original announcement of Theorem~\ref{I_is_on_the_cone}, in
February 2007 \cite{IHP}, the Gromov--Witten theory of toric
Deligne--Mumford stacks has matured considerably, and the proof that
we give here relies heavily on two recent advances.  The first is the
beautiful characterization of the Lagrangian cone $\sL_X$ for a toric
variety (or toric bundle) $X$ in terms of recursion relations
\cite{brown}; we establish the analogous result for toric
Deligne--Mumford stacks in Section~\ref{sec:toric_lag_cone}.  The
second is Liu's virtual localization formula for toric
Deligne--Mumford stacks \cite[Theorem~9.32]{ccliu}; this is the
essential technical ingredient that allows us to characterize $\sL_\X$
for a toric Deligne--Mumford stack $\X$.

A significant generalization of our Theorem~\ref{I_is_on_the_cone} has
recently been announced by Ciocan-Fontanine--Kim
\cite{cfk,ccfk}. Also one major application of our theorem, the
calculation of the quantum cohomology ring of smooth toric Deligne--Mumford
stacks with projective coarse moduli space, has been achieved directly
by Gonzalez and Woodward, using the theory of gauged Gromov--Witten
invariants \cite{gw1,gw2,Woodward:1,Woodward:2,Woodward:3}.  We feel that it is nonetheless
worth presenting our argument here, in part because it is based on
fundamentally different ingredients (on Givental's recursive
characterization of $\sL_X$, rather than on the the theory of
quasimaps or gauged Gromov--Witten theory), in part because it gives
explicit mirror formulas that have important applications, and in part
to reduce our embarrassment at the long gap between our announcement
of the mirror theorem and its proof.

The rest of this paper is organized as
follows. Sections~\ref{sec:GW_basics} and~\ref{sec:toric_stacks}
contain reviews of Gromov--Witten theory and toric Deligne--Mumford
stacks. In Section~\ref{sec:extended_Picard} we introduce a notion of
extended Picard group for a Deligne--Mumford stack. Our mirror
theorem, Theorem~\ref{I_is_on_the_cone}, is stated in Section
\ref{sec:toric_mir_thm}. In Section~\ref{sec:toric_lag_cone} we
establish a criterion for points to lie on the Lagrangian cone
$\sL_\X$. In Section~\ref{sec:pf_mir_thm} we prove
Theorem~\ref{I_is_on_the_cone} by showing that the extended
$I$-function satisfies the criterion from
Section~\ref{sec:toric_lag_cone}.

\begin{acknowledgements}
  We benefited greatly from discussions with Alexander Givental, Melissa
  Liu, and Yongbin Ruan.  T.C.~was supported in part by a Royal Society
  University Research Fellowship, ERC Starting Investigator Grant
  number~240123, and the Leverhulme Trust.  A.C.~was supported in part
  by EPSRC grants EP/E022162/1 and EP/I008128/1.  H.I.~was supported in
  part by EPSRC grant EP/E022162/1 and JSPS Grant-in-Aid for Scientific
  Research (C)~25400069.  H.-H.T.~was supported in part by a Simons
  Foundation Collaboration Grant.
\end{acknowledgements}

\section{Gromov--Witten Theory}\label{sec:GW_basics}
Gromov--Witten theory for orbifold target spaces was first constructed
in symplectic geometry by Chen--Ruan \cite{CR2}. In algebraic
geometry, the construction was established by
Abramovich--Graber--Vistoli \cite{AGV1, AGV2}. In this Section, we
review the main ingredients of orbifold Gromov--Witten theory. We
mostly follow the presentation of \cite{ts}. More detailed discussions
of the basics of orbifold Gromov--Witten theory from the viewpoint of
Givental's formalism can be found in e.g.~\cite{ts}, \cite{cit}.

\subsection{Chen--Ruan Cohomology}
\label{sec:Chen--Ruan}
Let $\X$ be a smooth Deligne--Mumford stack equipped 
with an action of an algebraic torus $\TT$. 
Let $X$ denote the coarse moduli space of $\X$. 
The inertia stack of $\X$ is defined
as 
\[
I\X:=\X\times_{\Delta, \X\times \X, \Delta}\X
\] 
where $\Delta \colon \X\to \X\times \X$ is the diagonal morphism.  A
point on $I\X$ is given by a pair $(x,g)$ of a point $x\in \X$ and an
element $g\in \Aut(x)$ of the isotropy group at $x$.  As a module over
$R_\TT := H^\bullet_\TT({\rm pt}, \CC)$, the $\TT$-equivariant {\em
  Chen--Ruan orbifold cohomology} of $\X$ is defined to be the
$\TT$-equivariant cohomology of the inertia stack:
\[
H_{\CR, \TT}^\bullet(\X):=H^\bullet_{\TT}(I\X,\CC) 
\]
When $\TT$ is the trivial group, this is denoted by
$H_{\CR}^\bullet(\X)$.  The work \cite{CR1} equips
$H_{\CR,\TT}^\bullet(\X)$ with a grading called the age grading and a
product called the Chen--Ruan cup product. These are different from
the usual ones on $H^\bullet_{\TT}(I\X,\CC)$.  There is an involution
$\inv \colon I\X \to I\X$ given on points by $(x,g) \mapsto
(x,g^{-1})$.  When the $\TT$-fixed set $\X^\TT$ is proper, we can
define the orbifold Poincar\'{e} pairing
\[
(\alpha,\beta)_\CR := \int_{I\X}^\TT \alpha \cup \inv^\star \beta
\]
on $H_{\CR,\TT}^\bullet(\X)$ using the Atiyah--Bott localization
formula; the pairing takes values in the fraction field $S_\TT$ of
$R_\TT = H_{\TT}^\bullet ({\rm pt})$.

\subsection{Gromov--Witten Invariants and Gromov--Witten Potentials}
Gromov--Witten invariants are intersection numbers in moduli stacks of
stable maps. Let $\Mbar_{g,n}(\X, d)$ denote the moduli stack of
$n$-pointed genus-$g$ degree-$d$ orbifold stable maps to $\X$ with 
sections to gerbes at the markings, where $d\in H_2(X, \ZZ)$ 
(see \cite[Section 4.5]{AGV2}, \cite[Section 2.4]{ts}). 
There are evaluation maps at the marked points 
\begin{align*}
  \ev_i \colon \Mbar_{g,n}(\X, d)\to I\X && 1\leq i\leq n
\end{align*}
and, given $\vec{b} = (b(1),\ldots,b(n))$ where the $b(i)$ correspond to
components $(I\X)_{b(i)}$ of $I\X$, we set:
\begin{align*}
  \Mbar_{g,n}^{\vec{b}}(\X, d) = \bigcap_{i=1}^n
  \ev_i^{-1} (I\X)_{b(i)}
  && \text{so that} &&
  \Mbar_{g,n}(\X, d) = \bigcup_{\vec{b}} 
  \Mbar_{g,n}^{\vec{b}}(\X, d)
\end{align*}
Let $\bpsi_i\in H^2(\Mbar_{g,n}(\X, d),\QQ), 1\leq i\leq n$, 
denote the descendant classes \cite[Section 2.5.1]{ts}. 
Suppose that $\Mbar_{g,n}(\X,d)$ is proper. 
Then the moduli stack carries a weighted virtual fundamental class 
\cite{AGV2},~\cite[Section 2.5.1]{ts}: 
\[ 
[\Mbar_{g,n}(\X, d)]^w\in H_\bullet(\Mbar_{g,n}(\X, d),\QQ)
\] 
Given elements $a_1,\ldots,a_n\in H^\bullet_\CR(\X)$ 
and nonnegative integers $k_1,\ldots,k_n$, we define:
\begin{equation}
  \label{eq:GW_invariant}
  \langle
  a_1\bpsi^{k_1},\ldots,a_n\bpsi^{k_n}
  \rangle_{g,n,d}:=\int_{[\Mbar_{g,n}(\X,
    d)]^w}(\ev_1^\star a_1)\bpsi_1^{k_1}\ldots (\ev_n^\star a_n)\bpsi_n^{k_n}
\end{equation}
These are called the 
{\em descendant Gromov--Witten invariants} of $\X$. 

When $\X$ is equipped with an action of an algebraic torus $\TT$,
there is an induced $\TT$-action on the moduli space
$\Mbar_{g,n,d}(\X,d)$.  The descendant classes $\bpsi_i$ and the
virtual fundamental class have canonical $\TT$-equivariant lifts and
we can define $\TT$-equivariant Gromov--Witten invariants
\[
\langle
  a_1\bpsi^{k_1},\ldots,a_n\bpsi^{k_n}
  \rangle_{g,n,d}^{\TT} 
  =\int_{[\Mbar_{g,n}(\X, d)]^w}^\TT 
(\ev_1^\star a_1)\bpsi_1^{k_1}\ldots (\ev_n^\star a_n)\bpsi_n^{k_n}
\]
for $a_1,\dots, a_n \in H_{\CR,\TT}^\bullet(\X)$.  In this paper we
consider the case where the moduli space $\Mbar_{g,n}(\X,d)$ itself
may not be proper, but the $\TT$-fixed locus is proper. This happens
for toric stacks.  In this case, we define $\TT$-equivariant
descendant Gromov--Witten invariants by using the virtual localization
formula (see \cite{ccliu}); the invariants then take values in $S_\TT
= \Frac(R_\TT)$.

We package descendant Gromov--Witten invariants using generating
functions.  Let $\bt=\bt(z)=t_0+t_1z+t_2z^2+\cdots\in
H^\bullet_{\CR}(\X)[z]$. Define:
\[
\langle\bt,\ldots,\bt\rangle_{g,n,d}=
\langle\bt(\bpsi),\ldots,\bt(\bpsi)\rangle_{g,n,d}
:=\sum_{k_1,\ldots,k_n\geq 0}\langle t_{k_1}\bpsi^{k_1},
\ldots,t_{k_n}\bpsi^{k_n}\rangle_{g,n,d} 
\]
The \emph{genus-$g$ descendant potential} of $\X$ is:
\[
\mathcal{F}_\X^g(\bt):=\sum_{n=0}^\infty \sum_{d\in
  \NE(\X)}\frac{Q^{d}}{n!}\langle\bt,\ldots,\bt\rangle_{g,n,d}
\]
Here $Q^d$ is an element of the {Novikov ring} $\Lambda_{\rm nov} :=
\CC[\![\NE(X)\cap H_2(X,\ZZ)]\!]$ (see \cite[Definition 2.5.4]{ts}),
where $\NE(X) \subset H_2(X,\RR)$ denotes the cone generated by
effective curve classes in $X$.  Let us fix an additive basis
$\{\phi_\alpha\}$ for $H^\bullet_\CR(\X)$ consisting of homogeneous
elements, and write
\begin{align*}
  t_k=\sum_\alpha t_k^\alpha\phi_\alpha\in H^\bullet_\CR(\X)
  && k\geq 0.
\end{align*}
The generating function $\mathcal{F}_\X^g(\bt)$ is a
$\Lambda_{\rm nov}$-valued formal power series in the variables
$t_k^\alpha$. 

The definition readily extends to the $\TT$-equivariant setting. 
The $\TT$-equivariant descendant 
Gromov--Witten potential $\mathcal{F}^g_{\X,\TT}(\bt)$ 
is defined as a 
$\Lambda_{\rm nov}^\TT := S_\TT[\![\NE(X)\cap H_2(X,\ZZ)]\!]$-valued 
function of $\bt(z) \in H_{\CR,\TT}^\bullet(\X)[z]$. 
Choosing a homogeneous basis $\{\phi_\alpha\}$ 
of $H_{\CR,\TT}^\bullet(\X)\otimes_{R_\TT} S_\TT$ over $S_\TT$, 
we write $\bt(z) = \sum_{k\ge 0} 
t_k z^k = \sum_{k\ge 0} \sum_{\alpha} t_k^\alpha \phi_\alpha z^k$.

\subsection{Givental's Symplectic Formalism}
\label{sec:Givental_formalism} 
Next we describe Givental's symplectic formalism for genus-zero
Gromov--Witten theory \cite{gi,gi2}.  We present the $\TT$-equivariant
version, following the presentation in \cite[Section 3.1]{ts},
\cite{cit} and \cite{ccit} for the non-equivariant case.

Since our target space $\X$ is not necessarily proper, we work over
the field $S_\TT \cong \CC(\chi_1,\dots,\chi_d)$ of fractions of
$H_{\TT}^\bullet({\rm pt})$, where $\{\chi_1,\dots,\chi_d\}$ is a
basis of characters of the torus $\TT\cong (\CC^\times)^d$.  Recall
that the $\TT$-equivariant Novikov ring is
\begin{equation} 
\label{eq:Tequiv_Novikov}
\Lambda_{\rm nov}^\TT = S_\TT[\![ \NE(X) \cap H_2(X,\ZZ)]\!] 
\end{equation} 
Givental's symplectic vector space is the $\Lambda_{\rm
  nov}^\TT$-module
\[ 
\sH := H_{\CR,\TT}(\X) \otimes_{R_\TT} 
S_\TT(\!(z^{-1})\!)[\![\NE(X) \cap H_2(X,\ZZ)]\!] 
\]
equipped with the symplectic form: 
\begin{align*}
  \Omega(f,g) = - \Res_{z=\infty} \big(f(-z), g(z)\big)_{\CR} \, dz 
  && \text{for $f,g\in \sH$.} 
\end{align*}
The coefficient of $Q^d$ in an element of $\sH$ is a formal Laurent
series in $z^{-1}$, i.e.~a power series of the form
$\sum_{n=n_0}^\infty a_n z^{-n}$ for some $n_0\in \ZZ$.  The
symplectic form $\Omega$ is given by the coefficient of $z^{-1}$ of
the orbifold Poincar\'{e} pairing $(f(-z), g(z))_{\CR}$; the minus
sign reflects the fact that we take the residue at $z=\infty$ rather
than $z=0$.
Consider the polarization
\begin{equation*}
  \sH=\sH_+\oplus\sH_-
\end{equation*} 
where 
\begin{align*} 
\sH_+ & :=H^\bullet_{\CR,\TT}(\X)\otimes_{R_\TT}  
S_\TT[z][\![\NE(X)\cap H_2(X,\ZZ)]\!] \\ 
\sH_-& :=z^{-1}H^\bullet_{\CR,\TT}(\X)\otimes_{R_\TT} S_\TT[\![z^{-1}]\!] 
[\![\NE(X) \cap H_2(X,\ZZ)]\!] 
\end{align*} 
The subspaces $\sH_{\pm}$ are maximally isotropic with respect to 
$\Omega$, and the symplectic form $\Omega$ induces a 
non-degenerate pairing between $\sH_+$ and $\sH_-$. 
Thus we can regard $\sH=\sH_+\oplus \sH_-$ as the 
total space of the cotangent bundle $T^*\sH_+$ of $\sH_+$. 

Let $\{\phi^\mu\}\subset H^\bullet_{\CR,\TT}(\X)\otimes_{R_\TT} S_\TT$ 
be the $S_\TT$-basis dual to $\{\phi_\nu\}$ 
with respect to the orbifold Poincar\'e pairing, 
so that $(\phi^\mu,\phi_\nu)_{\CR}=\delta^\mu_\nu$. 
A general point in $\sH$ takes the form
\begin{equation}
\label{eq:general_point_sH}
\sum_{a=0}^\infty\sum_\mu p_{a,\mu} \phi^\mu(-z)^{-a-1} +
\sum_{b=0}^\infty\sum_\nu q_b^\nu\phi_\nu z^b
\end{equation} 
and this defines Darboux co-ordinates $\{p_{a,\mu},q_b^\nu\}$ on
$(\sH,\Omega)$ which are compatible with the polarization $\sH = \sH_+
\oplus \sH_-$.  Put $p_a=\sum_\mu p_{a,\mu} \phi^\mu$, 
$q_b=\sum_\nu q_b^\nu \phi_\nu$, and denote:
\begin{equation*}
\begin{split}
&\bp=\bp(z):=\sum_{k=0}^\infty p_k(-z)^{-k-1}=
 p_0(-z)^{-1}+p_1(-z)^{-2}+\cdots\\
&\bq=\bq(z):=\sum_{k=0}^\infty q_kz^k=q_0+q_1z+q_2z^2+\cdots
\end{split}
\end{equation*}
We relate the co-ordinates $\bq$ on $\sH_+$ to the variables 
$\bt$ of the descendant potential $\mathcal{F}^g_\X(\bt)$ by 
$\bq(z) = \bt(z) - 1 z$; this identification is called the \emph{dilaton shift} 
\cite{gi}. 

The genus-zero descendant potential $\F_{\X,\TT}^0$ defines a formal germ of
a Lagrangian submanifold
\[
\sL_\X:=\big\{(\bp,\bq) \in \sH_+ \oplus \sH_- 
: \bp=d_\bq\F_{\X,\TT}^0\big\}\subset T^*\sH_+ \cong \sH 
\]
given by the graph of the differential of $\F_{\X,\TT}^0$. 
The submanifold $\sL_\X$ may be viewed as a formal subscheme 
of the formal neighbourhood of $-1z$ in $\sH$ cut out by the equations 
\[
p_{a,\mu}=\frac{\partial
  \F_{\X,\TT}^0}{\partial q_a^\mu}
\] 
Let $x=(x_1,\dots,x_m)$ be formal variables. Instead of giving a
rigorous definition of $\sL_\X$ as a formal scheme
(cf.~\cite[Appendix~B]{ccit}) we define the notion of a
$\Lambda_{\rm nov}^\TT[\![x]\!]$-valued point on $\sL_X$.  By a
\emph{$\Lambda_{\rm nov}^\TT[\![x]\!]$-valued point} of $\sL_\X$, we
mean an element of $\sH[\![x]\!]$ of the form
\begin{equation}
  \label{eq:very_big_J_function}
  - 1 z+ \bt(z)+
  \sum_{n=0}^\infty
  \sum_{d\in\NE(X)}
  \sum_{\alpha}
  \frac{Q^d}{n!}
  \<\bt(\bpsi),\ldots,\bt(\bpsi),\frac{\phi_\alpha}{-z-\bpsi}\>_{0,n+1,d}^\TT 
  \phi^\alpha
\end{equation}
for some $\bt(z)\in \sH_+[\![ x ]\!]$ satisfying 
\begin{equation} 
\label{eq:hot} 
\bt|_{x=Q=0} = 0 
\end{equation} 
Here the expression $\phi_\alpha/(-z-\bpsi)$ should be expanded as a
power series $\sum_{n=0}^\infty (-z)^{-n-1} \phi_\alpha \bpsi^n$ in
$z^{-1}$.  The condition \eqref{eq:hot} ensures that the expression
\eqref{eq:very_big_J_function} converges in the $(Q,x)$-adic topology.

\begin{rmk} 
\label{rmk:rational_Givental} 
As we shall see in Section~\ref{sec:toric_lag_cone}, using
localization in $\TT$-equivariant cohomology, the expression
\eqref{eq:very_big_J_function} lies in a rational version of
Givental's symplectic space:
\[
\sH_{\rm rat} := H_{\CR,\TT}^\bullet(\X) \otimes_{R_\TT} 
S_{\TT\times \CC^\times}[\![\NE(X) \cap H_2(X,\ZZ)]\!]  
\] 
where $S_{\TT \times \CC^\times} = 
\Frac(H_{\TT\times \CC^\times}^\bullet({\rm pt})) 
\cong \CC(\chi_1,\dots,\chi_d,z)$ and $z$ is identified 
with the $\CC^\times$-equivariant parameter. 
The space $\sH_{\rm rat}$ is embedded into $\sH$ by the 
Laurent expansion at $z=\infty$.  
This fact plays an important role in the characterization of points on 
$\sL_\X$ in Section~\ref{sec:toric_lag_cone}. 
\end{rmk} 

The Lagrangian submanifold $\sL_\X$ has very special 
geometric properties. 

\begin{thm}[\cite{gi2}, \cite{ccit}, \cite{ts}]
\label{thm:lag_cone}
  $\sL_\X$ is the formal germ of a Lagrangian cone with vertex at the
  origin such that each tangent space $T$ to the cone is tangent to
  the cone exactly along $zT$.
\end{thm}
In other words, if $N$ is a formal neighborhood in $\sH$ of $-1 z\in
\sL_\X$, then we have the following statements:
\begin{equation}\label{overruled}
\begin{split}
    &\text{(a) } T \cap \sL_\X = zT \cap N ;\\
    &\text{(b) } \text{for each } \fb \in zT \cap N , 
\text{ the tangent space to } \sL_\X \text{ at } \fb \text{ is } T ;\\
    &\text{(c) } \text{if } T = T_\fb \sL_\X \text{ then } \fb \in zT\cap N .
\end{split}
\end{equation}
Givental has proven that these statements are equivalent to the string
equation, dilaton equation, and topological recursion relations
\cite[Theorem~1]{gi2}.  The statements \eqref{overruled} imply that:
\begin{itemize}
\item the tangent spaces $T$ of $\sL_\X$ are closed under
  multiplication by $z$;
\item $\sL_\X$ is the union of the (finite-dimensional) family of
  linear subspace-germs
  \[
  \{zT\cap N : T \text{ is a tangent space of } \sL_\X \}
  \]
\end{itemize}

\begin{rmk}
\label{rmk:jfunction} 
A finite-dimensional slice of the Lagrangian submanifold $\sL_\X$ is
given by the so-called $J$-function {\cite{gi2}, \cite[Definition
  3.1.2]{ts}}
  \[
  J_\X(t,z)=1 z+t+
  \sum_{n=0}^\infty
  \sum_{d\in\NE(\X)}
  \sum_{\alpha}
  \frac{Q^d}{n!}
  \<t,\ldots,t,\frac{\phi_\alpha}{z-\bpsi}\>_{0,n+1,d}
  \phi^\alpha
  \]
which is a formal power series in co-ordinates $t^\alpha$ of 
 $t = \sum_\alpha t^\alpha \phi_\alpha \in H^\bullet_{\CR,\TT}(\X) 
 \otimes_{R_\TT}S_\TT$ taking values in $\sH$. 
The $J$-function $J_\X(t,-z)$ gives a 
$\Lambda_{\rm nov}^\TT[\![t]\!]$-valued point of the 
Lagrangian submanifold $\sL_\X$. 
\end{rmk}

\subsection{Twisted Gromov--Witten Invariants} 
\label{sec:twisted}
We will need also to consider Gromov--Witten invariants twisted by the
$\TT$-equivariant inverse Euler class \cite{cg,ts}.  In this section
we assume that the torus $\TT$ acts on the target space $\X$
trivially.  This is sufficient for our purposes, as in
Section~\ref{sec:toric_lag_cone} we consider twisted Gromov--Witten
theory for a $\TT$-fixed point of a toric stack.  Givental's
symplectic formalism for the twisted theory has a subtle but important
difference from that in the previous section: we need to work with
formal Laurent series in $z$ rather than $z^{-1}$.

Let $E \to \X$ be a vector bundle equipped with a $\TT$-linearization;
as mentioned above, $\TT$ here acts trivially on the base $\X$.
Consider the virtual vector bundle $E_{g,n,d} = R\pi_\star \ev^\star E
\in K^0_\TT(\Mbar_{g,n}(\X, d))$ where $\pi \colon \C_{g,n,d} \to
\Mbar_{g,n}(\X,d)$ and $\ev \colon \C_{g,n,d} \to \X$ give the
universal family of stable maps:
\[
\xymatrix{
  \C_{g,n,d} \ar[r]^-\ev \ar[d]_-\pi  & \X \\ 
  \Mbar_{g,n}(\X,d) 
}
\]
Let $\eT^{-1}(\cdot)$ denote the inverse 
of the $\TT$-equivariant Euler class. 
Twisted Gromov--Witten invariants 
\[
\langle a_1\bpsi^{k_1},\ldots,a_n\bpsi^{k_n} \rangle_{g,n,d}^{\eT^{-1},E}
\]
are defined by replacing the weighted virtual fundamental class
$[\Mbar_{g,n}(\X, d)]^w$ in \eqref{eq:GW_invariant}
by $[\Mbar_{g,n}(\X, d)]^w \cap
\eT^{-1}(E_{g,n,d})$.  
The \emph{twisted genus-$g$ descendant potential} is:
\[
\mathcal{F}_{\eT^{-1},E}^g(\bt):=\sum_{n=0}^\infty \sum_{d\in
  \NE(\X)}\frac{Q^{d}}{n!}\langle\bt,\ldots,\bt\rangle_{g,n,d}^{\eT^{-1},E}
\]
In the twisted theory, we work with the twisted orbifold Poincar\'{e}
pairing
\[
\big(\alpha,\beta\big)_\CR^{\eT^{-1},E} := \int_{I\X} \alpha \cup \inv^\star
\beta \cup \eT^{-1}(IE)
\]
where $IE$ is the inertia stack of the total space of $E$; $IE$ is a
vector bundle over $I\X$ such that the fibre over $(x,g) \in I\X$ is
the $g$-fixed subspace of $E_x$.  Givental's symplectic vector space
for twisted theory is the $\Lambda_{\rm nov}^\TT$-module
\[
\sH^\tw = H^\bullet_{\CR}(\X) \otimes S_\TT(\!(z)\!) 
[\![\NE(X) \cap H_2(X,\ZZ)]\!] 
\]
equipped with the symplectic form:
\[
\Omega(f,g) = \Res_{z=0} \big(f(-z), g(z)\big)_\CR^{\eT^{-1},E} dz 
\]
The polarization $\sH^\tw = \sH^\tw_+ \oplus \sH^\tw_-$ of $\sH^\tw$ 
is given by 
\begin{align*} 
\sH^\tw_+ & = H^\bullet_{\CR}(\X) \otimes S_\TT[\![z]\!] 
[\![ \NE(X) \cap H_2(X,\ZZ)]\!] \\ 
\sH^\tw_- & = H^\bullet_{\CR}(\X) \otimes S_\TT[z^{-1}]
[\![ \NE(X) \cap H_2(X,\ZZ)]\!] 
\end{align*} 

Let $\{\phi_\mu\}$, $\{\phi^\mu\}$ be dual bases of
$H_{\CR}^\bullet(\X)\otimes S_\TT$ with respect to the twisted
orbifold Poincar\'{e} pairing.  They define Darboux co-ordinates
$\{p_{a,\mu}, q_a^\mu \}$ on $\sH^\tw$ as in
\eqref{eq:general_point_sH}.  The Lagrangian submanifold $\sL_\tw$ of
the twisted theory is then defined similarly: a $\Lambda_{\rm
  nov}^\TT[\![x]\!]$-valued point of $\sL_\tw$ is an element of
$\sH^\tw[\![x]\!]$ of the form
\begin{equation}
  \label{eq:very_big_J_function_twisted}
  - 1 z+ \bt(z)+
  \sum_{n=0}^\infty
  \sum_{d\in\NE(X)}
  \sum_{\alpha}
  \frac{Q^d}{n!}
  \<\bt(\bpsi),\ldots,\bt(\bpsi),\frac{\phi_\alpha}{-z-\bpsi}
\>_{0,n+1,d}^{\eT^{-1},E}  
  \phi^\alpha
\end{equation}
for some $\bt(z)\in \sH_+^\tw[\![ x ]\!]$ satisfying $\bt|_{x=Q=0} =
0$.  Note that the expression \eqref{eq:very_big_J_function_twisted}
makes sense as an element of $\sH^\tw[\![x]\!]$.  We use here the fact
that, as $\TT$ acts trivially on $\X$, the descendant classes
$\bpsi_i$ are nilpotent on each moduli space $\Mbar_{0,n}(\X,d)$;
therefore $\bt(\bpsi) = \sum_{k=0}^\infty t_k \bpsi^k$ and
$\phi_\alpha/(-z-\bpsi) = \sum_{n=0}^\infty \phi_\alpha \bpsi^n
(-z)^{-n-1}$ truncate to finite series on each moduli space
$\Mbar_{0,n}(\X,d)$.

\begin{rmk}
The analogue of Theorem~\ref{thm:lag_cone} holds for $\sL_\tw$.  
\end{rmk}

\section{Toric Deligne--Mumford Stacks}\label{sec:toric_stacks}
In this Section we discuss some background material on toric stacks.
More details can be found in \cite{BCS, Iwanari1, Iwanari2, FMN}.

\subsection{Basics}\label{sec:toric_basics}
Following Borisov--Chen--Smith \cite{BCS}, a {\em toric
  Deligne--Mumford stack} is defined in terms of a stacky
fan 
\[
\bSigma=(N,\Sigma,\rho)
\]
where $N$ is a finitely generated abelian group, $\Sigma\subset
N_\QQ=N\otimes_{\ZZ}\RR$ is a rational simplicial fan, and $\rho
\colon \ZZ^{n} \to N$ is a homomorphism.  We denote by $\rho_i$ the
image under $\rho$ of the $i$th standard basis vector $e_i$ of
$\ZZ^n$.  Let $\LL\subset \ZZ^n$ be the kernel of $\rho$. The exact
sequence
\[
\xymatrix{
  0 \ar[r] & \LL \ar[r] & \ZZ^n \ar[r]^\rho &  N 
}
\]
is called the {\em fan sequence}.  By assumption, $\rho$ has finite
cokernel and the images $\bar{\rho}_i$, $1 \leq i \leq n$, of the
$\rho_{i}$s under the canonical map $N\to N_\QQ$ generate 
1-dimensional cones of the 
simplicial fan $\Sigma$.

By abuse of notation, we sometimes identify a cone $\sigma \in \Sigma$
with the subset $\{i :\bar{\rho}_i \in\sigma\}$ of $\{1,\dots,n\}$ and
write $i\in \sigma$ instead of $\bar{\rho}_i \in \sigma$.  The set of
{\em anti-cones} is defined to be
\begin{equation*}
\mathcal{A}:=\left\{I\subset \{1,\ldots,n\} : 
\sum_{i\notin I}\RR_{\geq 0} \bar{\rho}_i \text{ is a cone in } \Sigma \right\}.
\end{equation*}
Let 
\[
\U_\mathcal{A}:=\CC^n\setminus \bigcup_{I\notin \mathcal{A}}\CC^I
\]
where $\CC^I\subset \CC^n$ is the subvariety determined by the ideal
in $\CC[Z_1,\ldots,Z_n]$ generated by $\{Z_i : i\notin I\}$.  
Let ${\rho}^{\vee} \colon (\ZZ^*)^{n}\to \LL^{\vee}$ 
be the Gale dual of $\rho$ \cite{BCS}. 
Here $\LL^\vee :=H^1(\Cone(\rho)^*)$ is an extension of the ordinary 
dual $\LL^*=\Hom(\LL,\ZZ)$ by a torsion subgroup.  
We have $\Ker (\rho^\vee)=N^*$. 
The exact sequence 
\begin{equation} 
\label{eq:div_seq}
\xymatrix{
  0 \ar[r] & N^* \ar[r] & (\ZZ^*)^n \ar[r]^-{\rho^\vee} & \LL^\vee
}\end{equation} 
is called the {\em divisor sequence}.

Applying $\Hom_{\ZZ}(-,\Cstar)$ to $\rho^{\vee}$ gives a map 
\begin{equation}\label{eqn:torus}
  \alpha\colon G \to (\Cstar)^{n}.
\end{equation}
where $G :=\Hom_\ZZ(\LL^{\vee}, \Cstar)$. The toric
Deligne--Mumford stack $\X(\bSigma)$ associated to
$\bSigma$ is defined to be the quotient stack
\[
\X(\bSigma):=[\U_{\mathcal{A}}/G]
\]
where $G$ acts on $\U_{\mathcal{A}}$ via $\alpha$.

Throughout this paper we assume that the toric Deligne--Mumford stack
$\X(\bSigma)$ has {\em semi-projective} coarse moduli space, i.e.~that
the coarse moduli space $X(\Sigma)$ is a toric variety that has at
least one torus-fixed point, such that the natural map $X(\Sigma)\to
\Spec H^0\big(X(\Sigma), \mathcal{O}_{X(\Sigma)}\big)$ is projective.
In terms of the fan $\Sigma$, this is equivalent \cite{cls} to
demanding that the support $|\Sigma|$ of the fan $\Sigma$ is
full-dimensional and convex, and that there exists a strictly convex piecewise
linear function $\phi \colon |\Sigma| \to \RR$.

Let $N_{\tor}$ denote the torsion subgroup of $N$, and set
$\overline{N} := N/N_\tor$. For $c\in N$ we denote by $\overline{c}\in
\overline{N}$ the image of $c$ under the natural projection $N\to
\overline{N}$.  Given a stacky fan $\bSigma=(N,\Sigma,\rho)$, one can
consider the set $\Box$ defined as follows.  For a cone $\sigma\in
\Sigma$, define
\[
\Box(\sigma):=\left\{b\in N : \text{$\bar{b}=\sum_{i \in \sigma}
  a_i\bar{\rho}_i$ for some $a_i$ with $0\leq a_i< 1$} \right\}
\]
and set $\Box(\bSigma):=\bigcup_{\sigma\in \Sigma}
\Box(\sigma)$. 
Components of the inertia stack $I \X(\bSigma)$
are indexed by $\Box$; we write $I\X(\bSigma)_b$ for the component 
corresponding to $b\in \Box$.  
The involution $\inv$ on $I\X(\bSigma)$ 
induces an involution $b \mapsto \hat{b}$ on $\Box(\bSigma)$. 

Each cone $\sigma \in \Sigma$ defines a closed toric 
substack $\X(\bSigma)_\sigma \cong \X(\bSigma/\sigma)$, 
where $\bSigma/\sigma$ 
denotes the quotient stacky fan \cite[\S 4]{BCS} 
defined on the quotient space $N(\sigma) = 
N/\sum_{i\in \sigma} \ZZ \rho_i$. 
The component $I\X(\bSigma)_b$ of the inertia stack corresponding to 
$b\in \Box(\bSigma)$ is isomorphic to the toric substack 
$\X(\bSigma)_{\sigma(b)}$, 
where $\sigma(b)$ is the minimal cone containing $\bar{b}$.

\subsection{Extended Stacky Fans}

Following Jiang \cite{Jiang}, toric Deligne--Mumford stacks can also
be described using extended stacky fans. Let $\bSigma=(N,\Sigma,\rho)$
be a stacky fan, and let $S$ be a finite set equipped with a
map\footnote{The reader can keep in mind the most basic case where $S$
  is a subset of $\Box(\bSigma)$.}  $S \to N_\Sigma := \{ c\in N :
\bar{c} \in |\Sigma|\}$.  We label the finite set $S$ by
$\{1,\dots,m\}$, where $m=|S|$, and write $s_j\in N$ for the image of
the $j$th element of $S$.  The {\em $S$-extended stacky fan} is given
by the same group $N$, the same fan $\Sigma$, and the fan map
$\rho^S\colon\ZZ^{n+m}\to N$ defined by:
\begin{equation*}
\rho^S(e_i) = 
\begin{cases}
  \rho_i &  1\leq i\leq n \\
  s_{i-n} & n +1 \le  i \leq n+m
\end{cases}
\end{equation*}
Given an $S$-extended stacky fan $(N, \Sigma, \rho^S)$, an associated
stack may be defined as follows. 
Define 
\[
\U_{\mathcal{A},S} := \U_{\mathcal{A}}  \times (\CC^\times)^m. 
\]
Let $\LL^S$ be the kernel of
$\rho^S \colon \ZZ^{n+m}\to N$. 
Applying Gale duality to the $S$-extended fan sequence 
$0\to \LL^S\to \ZZ^{n+m}\to N$
gives the $S$-extended divisor sequence:
\[
\xymatrix{
  0 \ar[r] & N^* \ar[r] & (\ZZ^*)^{n+m} \ar[r]^-{\rho^{S\vee}} & \LL^{S\vee}
}
\]
Applying $\Hom_\ZZ(-,\Cstar)$ to the $S$-extended divisor
sequence gives a map $\alpha^S \colon G^S \to (\Cstar)^{n+m}$ where
$G^S :=\Hom_\ZZ(\LL^{S\vee}, \Cstar)$. We consider the
quotient stack
\begin{equation} 
\label{eq:extended_quotient} 
\big[\U_{\mathcal{A},S}/G^S\big]
\end{equation} 
where $G^S$ acts on $\U_{\mathcal{A},S}$ via $\alpha^S$. 
Jiang showed \cite{Jiang} that this stack associated to 
the $S$-extended stacky fan 
$(N, \Sigma, \rho^S)$ is isomorphic to the stack
$\X(\bSigma)$.  

\subsection{Torus Action and Line Bundles} 
\label{sec:equiv_obj} 
The inclusion $(\CC^\times)^n \subset \U_{\mathcal{A}}$ induces an
open embedding of the Picard stack $\T = [(\CC^\times)^n/G]$ into
$\X(\bSigma)$.  We have $\T \cong \TT \times BN_{\rm tor}$ with $\TT
:= (\CC^\times)^n/\im \alpha \cong N\otimes \CC^\times$ and $N_{\rm
  tor} \cong \Ker \alpha$, where $\alpha$ is given in
\eqref{eqn:torus}.  The Picard stack $\T$ acts naturally on
$\X(\bSigma)$ and the $\T$-action restricts to the $\TT$-action on
$\X(\bSigma)$.

A line bundle on $\X(\bSigma)$ corresponds to a $G$-equivariant 
line bundle on $\U_{\mathcal{A}}$, 
and a $\T$-equivariant line bundle on $\X(\bSigma)$ corresponds to 
a $(\CC^\times)^n$-equivariant line bundle on $\U_{\mathcal{A}}$. 
Thus we have natural identifications: 
\begin{align*} 
\Pic(\X(\bSigma)) & \cong \Hom(G,\CC^\times) \cong \LL^\vee \\ 
\Pic_{\T}(\X(\bSigma)) & \cong \Hom((\CC^\times)^n, \CC^\times) 
\cong (\ZZ^n)^*
\end{align*} 
The natural map $\Pic_{\T}(\X(\bSigma)) \to \Pic(\X(\bSigma))$ is
identified with the divisor map $\rho^\vee \colon (\ZZ^n)^* \to
\LL^\vee$ in \eqref{eq:div_seq}.  We write $u_1,\dots,u_n$ for the
basis of $\T$-equivariant line bundles on $\X(\bSigma)$ corresponding
to the standard basis of $(\ZZ^n)^*$ and write $D_1,\dots,D_n$ for the
corresponding non-equivariant line bundles, i.e.~$D_i =
\rho^\vee(u_i)$.  Abusing notation, we also write $u_i$ or $D_i$ for
the corresponding ($\TT$-equivariant or non-equivariant) first Chern
classes. These are the ($\TT$-equivariant or non-equivariant)
Poincar\'{e} duals of the toric divisors $[\{Z_i=0\}/G] \subset
[\U_{\mathcal{A}}/G]$.

\subsection{Chen--Ruan Cohomology}
\label{sec:CR_coh}
The Chen--Ruan orbifold cohomology (see Section~\ref{sec:Chen--Ruan})
of the toric Deligne--Mumford stack $\X(\bSigma)$ associated
to a stacky fan $\bSigma=(N,\Sigma,\rho)$ has been computed
by Borisov--Chen--Smith \cite{BCS} and, in the semi-projective case,
by Jiang--Tseng \cite{JT}:
\begin{equation*}
H^\bullet_{\CR}(\X(\bSigma), \CC)\simeq \frac{\CC[N_{\Sigma}]}
{\left\{\sum_{i=1}^n \chi(\rho_i)y^{\rho_i} : \chi \in N^*\right\}}
\end{equation*}
where:
\begin{enumerate}
\item
$\CC[N_{\Sigma}]:=\bigoplus_{c\in N_\Sigma} \CC y^{c}$ with the product:
\[ 
y^{c_1}\cdot y^{c_2}:= 
\begin{cases}
  y^{c_1+c_2} & \text{if there is a cone $\sigma\in \Sigma$ 
such that $\overline{c_1}$,~$\overline{c_2}\in \sigma$;}\\
  0 & \text{otherwise.}
\end{cases}
\] 
\item 
$N_\Sigma:=\{c\in N : 
\text{ $\overline{c}\in \sigma $ for some $\sigma \in \Sigma$ }\}$
\end{enumerate}
Similarly, the $\TT$-equivariant Chen--Ruan orbifold cohomology
of the toric Deligne--Mumford stack $\X(\bSigma)$ is \cite{ccliu}:
\begin{equation*}
H^\bullet_{\CR, \TT}(\X(\bSigma), \CC)\simeq 
\frac{R_{\TT}[N_{\Sigma}]}{
\left\{\chi - \sum_{i=1}^n \chi(\rho_i)y^{\rho_i} : 
\chi \in N^* \otimes \CC \cong H^2_{\TT}({\rm pt}) \right \}},
\end{equation*}
where:
\begin{enumerate}
\item 
$R_{\TT}:=H_{\TT}^\bullet(\text{pt}) = 
\Sym^\bullet_\CC(N^*\otimes \CC)$; 
\item
$R_{\TT}[N_{\Sigma}]:=\bigoplus_{c\in N_\Sigma} R_{\TT} y^{c}$ 
with the product:
\[
y^{c_1}\cdot y^{c_2}:= 
\begin{cases}
y^{c_1+c_2} & \text{if there is a cone $\sigma\in \Sigma$ 
such that $\overline{c_1}$,~$\overline{c_2}\in \sigma$;}\\ 
0 & \text{otherwise.}
\end{cases}
\]
\end{enumerate}
The ($\TT$-equivariant or non-equivariant) classes $u_i$, $D_i$ 
in Section~\ref{sec:equiv_obj} 
correspond to $y^{\rho_i}$ in the above descriptions. 
For $b\in \Box(\bSigma)$, $y^b$ is the identity class supported 
on the twisted sector $I\X(\bSigma)_b$. 

\subsection{Maps to $1$-Dimensional Torus Orbits}

We next describe toric maps from certain very simple toric orbifolds
$\PP_{r_1,r_2}$ to the toric Deligne--Mumford stack
$\X(\bSigma)$.  This establishes notation that we will need to
state and prove our mirror theorem.

\begin{defn}
  Let $r_1$ and $r_2$ be positive integers. There is a unique
  Deligne--Mumford stack with coarse moduli space equal to $\PP^1$,
  isotropy group $\mu_{r_1}$ at $0 \in \PP^1$, isotropy group
  $\mu_{r_2}$ at $\infty \in \PP^1$, and no other non-trivial isotropy
  groups.  We call this stack $\PP_{r_1,r_2}$.
\end{defn}

Let $r = \lcm(r_1,r_2)$, and let $r_1'$ and $r_2'$ satisfy $r_1 r_2' =
r_1' r_2 = r$ (i.e.~$r_i' = r_i/\gcd(r_1,r_2)$). 
The stack $\PP_{r_1,r_2}$ is a toric Deligne--Mumford stack with fan sequence:
\[
\xymatrix{
  0 \ar[r] & \ZZ \ar[rr]^
  {\begin{pmatrix}
      r_2^\prime\\
      r_1^\prime
    \end{pmatrix}
  } &&
  \ZZ^2\ar[rr]^
  {
    \begin{pmatrix}
      -r_1,r_2
    \end{pmatrix}
  } && \ZZ
}
\]

\begin{prop}[{cf.~\cite[Lemma~2]{Johnson}}]
  \label{prop:one-dimensional}
  Let $\X(\bSigma)$ be the toric Deligne--Mumford stack associated to a stacky
  fan $\bSigma=(N,\Sigma,\rho)$, and suppose that the fan
  $\Sigma$ is complete and $1$-dimensional.  
Let $\sigma_1 = \langle \bar{\rho}_1 \rangle$, 
$\sigma_2 = \langle \bar{\rho}_2 \rangle$ be
  the 1-dimensional cones of $\Sigma$, and assume without loss of
  generality that $\bar{\rho}_1 < 0$ and $\bar{\rho}_2 > 0$ 
in $N_\QQ\simeq \QQ$. 
Let $w_2 \rho_1 + w_1 \rho_2 =0$ with $w_1,w_2 \in \ZZ_{>0}$ 
be the minimal integral relation between $\rho_1,\rho_2 \in N$. 
The following are equivalent:
  \begin{itemize}
  \item[(a)] A representable toric morphism $f\colon \PP_{r_1, r_2}\to \X(\bSigma)$
    for some $r_1$,~$r_2$ such that $f(0) =\X(\bSigma)_{\sigma_1}$ and 
    $f(\infty) = \X(\bSigma)_{\sigma_2}$; 
  \item[(b)] Two box elements $b_1 \in \Box(\sigma_1)$, $b_2 \in
    \Box(\sigma_2)$ and non-negative integers $q_1$,~$q_2$ such that
    $q_1\rho_1+q_2\rho_2+b_1+b_2=0$ in $N$; 
  \item[(c)] A box element $b_1 \in \Box (\sigma_1)$ and a strictly positive
    rational number $l$ such that $w_2l-f_1$ is a non-negative
    integer, where $\overline{b}_1=f_1\overline{\rho}_1$.
  \end{itemize}
These data are related as follows: $r_i$ is the order of $b_i$ in $N/\ZZ\rho_i$; 
$l= (q_2 + f_2)/w_1 = (q_1 + f_1)/w_2$; 
$q_1 = \lfloor l w_2 \rfloor$, $q_2 = \lfloor l w_1 \rfloor$, 
$f_1 = \langle lw_2 \rangle$, $f_2 = \langle l w_1 \rangle$.  

\end{prop}

\begin{proof}
  Let
  \begin{equation}
    \label{eq:fan_sequence_for_X}
    \begin{aligned}
      \xymatrix{
        0\ar[rr]& &\ZZ \ar[rr]^{
          \begin{pmatrix}
            w_2\\
            w_1
          \end{pmatrix}}
        & &\ZZ^2 \ar[rr]^\rho  & &N}
    \end{aligned}
  \end{equation}
  be the fan sequence for $\X(\bSigma)$.  A representable toric morphism
  $f\colon \PP_{r_1,r_2}\to \X(\bSigma)$ is given by a commutative diagram:
  \begin{equation}
    \label{eq:toric_morphism}
    \begin{aligned}
      \xymatrix{
        0 \ar[rr] & & \ZZ \ar[dd]_m \ar[rr]^{
          \begin{pmatrix}
            r_2^\prime\\
            r_1^\prime
          \end{pmatrix}
        }& &\ZZ^2\ar[dd]_{
          \begin{pmatrix}
            m_1 & 0\\
            0   & m_2
          \end{pmatrix}
        } \ar[rr]^{
          \begin{pmatrix}
            -r_1 & r_2
          \end{pmatrix}
        }& &\ZZ\ar[dd]^\eta\\
        & & & & & &  \\
        0\ar[rr]& &\ZZ \ar[rr]_{
          \begin{pmatrix}
            w_2\\
            w_1
          \end{pmatrix}}
        & &\ZZ^2 \ar[rr]_\rho  & &N}
    \end{aligned}
  \end{equation}
  for some integers $m_1$,~$m_2$,~$m$ and some map $\eta$.  Given a
  morphism as in (a), and hence a commutative diagram
  \eqref{eq:toric_morphism}, let $b_1$ be the unique element of
  $\Box(\sigma_1)$ such that $b_1 \equiv \eta(-1) \mod \langle \rho_1
  \rangle$, and let $b_2$ be the unique element of $\Box(\sigma_2)$
  such that $b_2 \equiv \eta(1) \mod \langle \rho_2\rangle$.  Then
  there exist unique non-negative integers $q_1$,~$q_2$ such that:
  \begin{align*}
    \eta(-1) & =  q_1 \rho_1 + b_1 &
    \eta(1) & =  q_2 \rho_2 + b_2
  \end{align*}
  and we have $q_1\rho_1+q_2\rho_2+v_1+v_2=0$ in $N$.  Thus a morphism
  as in (a) determines data as in (b).
  
  Conversely, suppose that we are given $v_1$,~$v_2$,~$q_1$,~$q_2$ as
  in (b).  Define $\eta\colon \ZZ \to N$ by setting:
  \[
  \eta(1)=-b_1-q_1\rho_1=b_2+q_2\rho_2.
  \]
  Now set $r_i=\ord b_i$ in the group $N/\rho_i$; by definition there
  are integers $k_1$, $k_2$ such that $r_ib_i=k_i\rho_i$, and---for
  instance by looking at images in $\overline{N}$---we see that $0\leq
  k_i< r_i$.  Now set
  \begin{align*}
    m_1=r_1q_1+k_1 &&
    m_2=r_2q_2+k_2
  \end{align*}
  The diagram:
  \[
  \xymatrix{
    \ZZ\ar[rr]^{
      \begin{pmatrix}
        r_1^\prime\\
        r_2^\prime
      \end{pmatrix}
    }  
    & & 
    \ZZ^2\ar[dd]_{
      \begin{pmatrix}
        m_1 & 0\\
        0   & m_2
      \end{pmatrix}
    } \ar[rr]^{
      \begin{pmatrix}
        -r_1 & r_2
      \end{pmatrix}
    }& &\ZZ\ar[dd]^\eta\\
    & & & &  \\
    & & \ZZ^2 \ar[rr]_\rho  & &N}
\]
is commutative: $m_1\rho_1=r_1q_1\rho_1+k_1\rho_1=r_1\bigl(-b_1+
\eta(-1)\bigr)+r_1b_1=\eta(-r_1)$, and similarly $m_2 \rho_2 =
\eta(r_2)$.  Thus:
\[
\begin{pmatrix}
  m_1r_2^\prime\\
  m_2r_1^\prime
\end{pmatrix}
\in \Ker \rho 
\]
The fan sequence \eqref{eq:fan_sequence_for_X} defining $\X(\bSigma)$ is exact
at $\ZZ^2$, and we deduce that there exists an integer $m>0$ such that
\[
\begin{pmatrix}
  m_1r_2^\prime\\
  m_2r_1^\prime
\end{pmatrix}
=
\begin{pmatrix}
  w_2m\\
  w_1m
\end{pmatrix}
\]
and hence that the diagram:
\[
\xymatrix{
  0 \ar[rr] & & \ZZ \ar[dd]_m \ar[rr]^{
    \begin{pmatrix}
      r_2^\prime\\
      r_1^\prime
    \end{pmatrix}
  }& &\ZZ^2\ar[dd]_{
    \begin{pmatrix}
      m_1 & 0\\
      0   & m_2
    \end{pmatrix}
  } \ar[rr]^{
    \begin{pmatrix}
      -r_1 & r_2
    \end{pmatrix}
  }& &\ZZ\ar[dd]^\eta\\
  & & & & & &  \\
  0\ar[rr]& &\ZZ \ar[rr]_{
    \begin{pmatrix}
      w_2\\
      w_1
    \end{pmatrix}}
  & &\ZZ^2 \ar[rr]_\rho  & &N}    
\]
defines a stable representable morphism 
$f\colon \PP_{r_1,r_2}\to \X(\bSigma)$.

It is almost immediate that the constructions $(a)\Rightarrow (b)$ and
$(b)\Rightarrow (a)$ are inverses of each other: the key point is
that, if $f\colon \PP_{r_1,r_2}\to \X(\bSigma)$ is representable, then $r_i$ is
the order of $b_i$ in $N/\rho_i$.

The equivalence $(b) \Leftrightarrow (c)$ is immediate: 
we set $q_1 =w_2l-f_1$, write $w_1 l=q_2+f_2$ 
with $f_2=\langle w_1l \rangle$ the
fractional part and $q_2=\lfloor w_1l\rfloor$ the integer part, and
set $b_2=-q_1 \rho_1 -q_2\rho_2-b_1$.
\end{proof}

\begin{rmk} 
The box elements $b_1, b_2$ in the above proposition 
are given by the restrictions of $f$ to $0,\infty \in \PP_{r_1,r_2}$ respectively. 
The rational number $l>0$ in (c) measures the ``degree" of the map $f$ 
in the sense that $l = \int_{\PP_{r_1,r_2}} c_1(f^* \sO(1)) = m/\lcm(r_1,r_2)$,
where $\sO(1)$ is the positive generator of 
$\Pic(\X(\bSigma))$ modulo torsion. 
The degree of the map between the coarse moduli spaces 
$\bar{f} \colon \PP^1 \cong |\PP_{r_1,r_2}| \to \PP^1 \cong X(\Sigma)$ is 
given by $\overline{\eta(1)} = (q_2 \bar{\rho}_2 + \bar{b}_2) 
= (q_2 + f_2) \bar{\rho}_2 = l w_1 \bar{\rho}_2 = lw_2 |\bar{\rho}_1| 
\in \ZZ$. 
\end{rmk} 

\begin{notation}
  \label{notation:what_is_j}
  Let $\bSigma=(N,\Sigma,\rho)$ be a stacky fan.  We write
  $\sigma | \sigma'$ if $\sigma$,~$\sigma' \in \Sigma$ are
  top-dimensional cones that meet along a codimension-1 face.
  Whenever $\sigma|\sigma'$, we write $j$ for the unique index such
  that $\bar{\rho}_j$ is in $\sigma$ but not in $\sigma'$, and $j'$
  for the unique index such that $\bar{\rho}_{j'}$ is in $\sigma'$ but
  not in $\sigma$.
\end{notation}

\begin{notation}
  Let $\bSigma=(N,\Sigma,\rho)$ be a stacky fan. Given $b\in
  \Box(\bSigma)$, we define $b_i \in [0,1)$, $1\le i\le n$ 
by the conditions: 
\begin{itemize} 
\item $\bar{b}=\sum_{i=1}^n b_i\bar{\rho}_i$; 
\item $b_i = 0$ if $i \notin \sigma(b)$, where 
$\sigma(b)\in \Sigma$ is the minimal cone containing $\bar{b}$. 
\end{itemize} 
\end{notation}

\begin{prop}
  \label{prop:what_is_c}
  Let $\X(\bSigma)$ be the toric Deligne--Mumford stack associated to a stacky
  fan $\bSigma=(N,\Sigma,\rho)$.  Suppose that
  $\sigma$,~$\sigma' \in \Sigma$ satisfy $\sigma|\sigma'$ 
and let $b \in \Box(\sigma)$. 
The following are equivalent:
  \begin{enumerate}
  \item A representable toric morphism $f\colon \PP_{r_1,r_2} \to \X(\bSigma)$
    such that $f(0) = \X(\bSigma)_\sigma$, $f(\infty) =
    \X(\bSigma)_{\sigma'}$ and the restriction 
   $f|_0\colon B\mu_{r_1} \to \X(\bSigma)_\sigma$ 
   gives the box element $\hat{b} \in \Box(\sigma)$. 
  \item A positive rational number $c$ such that $\langle c \rangle =
    \hat{b}_j$.
  \end{enumerate}
  Here $j$ is as in Notation~\ref{notation:what_is_j}, and $\hat{b}
  = \inv(b)$ (see Section~\ref{sec:Chen--Ruan}).
\end{prop}

\begin{proof}
  Let 
  \[
  w_j \rho_j + 
  \left( \sum_{i \in \sigma \cap \sigma'} w_i \rho_i \right) +
  w_{j'} \rho_{j'} = 0
  \]
  be the minimal integral relation between $\{ \rho_i : \bar{\rho}_i
  \in \sigma \cup \sigma' \}$ such that $w_j>0$.  Replacing $\Sigma$
  by $\Sigma/\tau$, we reduce to the case where $\X(\Sigma)$ is
  one-dimensional.  The result now follows from
  Proposition~\ref{prop:one-dimensional}, with $w_1$ there equal to
  $w_{j'}$ here, $w_2$ there equal to $w_j$ here, and $c$ equal to $l
  w_j$.
\end{proof}

\begin{rmk}
  \label{rmk:determined} 
Note that the choice of $\sigma$,~$\sigma'$,~$b$ and~$c$ in
Proposition~\ref{prop:what_is_c} determines the map $f \colon
\PP_{r_1,r_2} \to \X(\bSigma)$ uniquely, and hence determines both $r_2$ and
the box element $b' \in \Box(\sigma')$ given by the restriction 
$f|_\infty \colon B\mu_{r_2} \to \X(\bSigma)_{\sigma'}$. 
More precisely, $b'$ is the unique element of $\Box(\sigma')$ 
such that 
\begin{equation}
\label{eq:def_b'}
\hat{b} + \lfloor c \rfloor \rho_j + q' \rho_{j'} + b' \equiv 0 
\mod 
\bigoplus_{i\in \sigma\cap \sigma'} \ZZ \rho_i 
\end{equation} 
for some $q'\in \ZZ_{\ge 0}$.  Note the asymmetry between $b$ and
$b'$: the restriction $f|_0$ gives $\hat{b} = \inv(b)$ and the
restriction $f|_\infty$ gives $b'$.  This convention is useful in our
recursion analysis.
\end{rmk}

\begin{defn} 
\label{defn:what_is_l_c_j}
Let $\sigma$, $\sigma'\in \Sigma$ be top-dimensional cones satisfying
$\sigma|\sigma'$.  Let $j, j'$ be as in
Notation~\ref{notation:what_is_j}.  Define $l(c,\sigma,j)$ to be the
element of $\LL\otimes \QQ \cong H_2(X,\QQ)$ given by the unique
relation of the form:
\[
c \bar{\rho}_j + \left( \sum_{i\in \sigma\cap \sigma'} c_i \bar{\rho}_i\right) 
+ c' \bar{\rho}_{j'} =0 
\]
\end{defn}

\begin{rmk}
  When we have a box element $b\in \Box(\sigma)$ satisfying $\langle c
  \rangle = \hat{b}_j$, $l(c,\sigma,j)$ is the degree of the
  representable toric morphism $f\colon \PP_{r_1,r_2} \to \X(\bSigma)$
  specified by a rational number $c>0$ in
  Proposition~\ref{prop:what_is_c}(2).  We have $D_j \cdot
  l(c,\sigma,j) = c$, $D_{j'} \cdot l(c,\sigma,j) = c'$, $D_i \cdot
  l(c,\sigma,j) = c_i$ for $i\in \sigma\cap \sigma'$ and $D_i \cdot
  l(c,\sigma,j) =0$ for $i\notin \sigma \cup \sigma'$.
\end{rmk}

\begin{defn} 
\label{defn:fixed_curve_deg}
Let $\sigma$, $\sigma'\in \Sigma$ be top-dimensional cones satisfying
$\sigma|\sigma'$.  Let $j$,~$j'$ be as in
Notation~\ref{notation:what_is_j}.  Let $b\in \Box(\sigma)$ and $b'
\in \Box(\sigma')$.  Define $\Lambda E_{\sigma,b}^{\sigma',b'}\subset
\LL\otimes \QQ \cong H_2(\X,\QQ)$ to be the set of degrees of
representable toric morphisms $f\colon \PP_{r_1,r_2} \to \X(\bSigma)$
such that $f(0) = \X(\Sigma)_{\sigma}$, $f(\infty) =
\X(\Sigma)_{\sigma'}$ and $f|_0$ and $f|_\infty$ give respectively the
box elements $\hat{b}$ and $b'$.  In other words:
\[
\Lambda E_{\sigma,b}^{\sigma',b'} 
= \left\{ l(c,\sigma,j)\in \LL\otimes \QQ : 
\begin{array}{l} 
c>0 \text{ such that  $\langle c \rangle = \hat{b}_j$ and} \\ 
\text{that \eqref{eq:def_b'} holds 
for some $q'\in\ZZ_{\ge 0}$} 
\end{array} \right \}
\]
\end{defn}

\section{The Extended Picard Group}
\label{sec:extended_Picard}

In this Section we introduce notions of extended Picard group for a
Deligne--Mumford stack $\X$ and extended degree for an orbifold stable
map $f \colon\C \to \X$.  There is less here than meets the eye: the
extended degree of $f$ amounts in the end to a convenient way of
packaging the extra discrete data attached to $f$, given by the
elements of $\Box(\X)$ associated to the marked points.  In what
follows we will use this material only when $\X$ is a toric
Deligne--Mumford stack, but the definitions make sense for general
Deligne--Mumford stacks and we give them in this context.  

\begin{defn}
  The \emph{box} of a Deligne--Mumford stack $\X$, denoted $\Box \X$,
  is the set of generic representable morphisms $b \colon B\mu_r\to \X$. 
In other words, it is the set of connected components 
of the inertia stack $I\X$. 
We write the \emph{order} $r$ of the box element $b$ as $r_b$.
\end{defn}

\begin{rmk}
  If $\X$ is a toric Deligne--Mumford stack then this reduces to the
  notion of $\Box(\bSigma)$ given in Section~\ref{sec:toric_basics}.
\end{rmk}

\begin{defn}
  Let $\X$ be a Deligne--Mumford stack and let $S$ be a finite set
  equipped with a map $S\to \Box \X$.  Abusing notation, we denote an
  element of $S$ and its image in $\Box \X$ by the same symbol $b$.
  The \emph{$S$-extended Picard group} of $\X$, denoted by $\Pic^S
  \X$, is defined by the exact sequence:
  \begin{equation} 
    \label{eq:PicS_seq} 
    \xymatrix{
      0 \ar[r] & \Pic^S \X \ar[r] & \Pic \X \oplus \bigoplus_{b\in S} r_b^{-1}\ZZ \ar[r] &  {\bigoplus_{b\in S}} r_b^{-1} \ZZ/\ZZ \ar[r] & 0
    }
  \end{equation}
  In other words, an element of $\Pic^S \X$ is a pair $(L,\varphi)$ where
  $L\in \Pic \X$ is a line bundle on $\X$, and $\varphi\colon S \to \QQ$ has
  the property that $\varphi(b) +  \age_b (L) \in \ZZ$, 
  where $\age_b(L)$ is the age of $L$ at $b$, i.e.~$\age_b(L) = 
  k_b/r_b$ with $0\le k_b<r_b$ the character of the $\mu_{r_b}$-representation 
$b^\star L$. 
\end{defn}

\begin{defn} 
\label{def:extended_degree}
Let $f\colon (\C,x_1,\dots,x_k) \to \X$ be an orbifold stable map. 
An \emph{$S$-decoration} of $f$ is an assignment of 
$s_j\in S$ to each marking $x_j$ such that the element of 
$\Box \X$ given by $f|_{x_j}$ coincides with the image of $s_j$ in $\Box \X$. 
The \emph{$S$-extended degree} of an $S$-decorated orbifold stable 
map $f$ is an element of $(\Pic^S \X)^*$ defined by 
\[
\deg^S(f)(L,\varphi) = \deg f^*L + \sum_{j=1}^k \varphi(s_j) 
\]
The Riemann--Roch theorem for orbifold curves \cite{AGV2} shows that 
the right-hand side is an integer.   
The \emph{$S$-extended Mori cone} is the cone $\NE^S(\X) \subset 
(\Pic^S \X)^* \otimes \RR$ generated by the $S$-extended 
degrees of $S$-decorated orbifold stable maps. 
One can easily see that: 
\[
\NE^S(\X) \cong \NE(\X) \times \RR_{\ge 0}^{|S|} 
\]
under the standard decomposition 
\begin{equation} 
\label{eq:PicSdual_Picdual} 
(\Pic^S \X)^* \otimes \RR 
\cong ((\Pic \X)^* \otimes \RR) \oplus \RR^{|S|}
\end{equation} 
induced from \eqref{eq:PicS_seq},  
where $\NE(\X)$ denotes the usual Mori cone. 
\end{defn} 
\begin{rmk} 
We can think of elements of $S$ as ``states" to be 
inserted at markings of a stable map. 
If an $S$-decorated orbifold stable map has a degree $d\in H_2(X, \ZZ)$ 
and each ``state" $b\in S$ is inserted $n_b$ times to it, the $S$-extended 
degree with respect to $(L,\varphi)$ is given by: 
\[
\int_d c_1(f^*L)   + \sum_{b\in S} n_b \varphi(b). 
\]
The value $\varphi(b)$ can be viewed as the degree of 
the variable dual to $b\in S$. 
\end{rmk} 

\begin{rmk}
  When $\X$ is Gorenstein and the subset $S$ consists of those box
  elements of age $1$, the $S$-extended degree of a stable map is
  essentially the same thing as the orbifold Neron--Severi degree
  defined by Bryan--Graber \cite[\S2]{BG}.
\end{rmk}

\subsection{Extended Degrees for Toric Stacks}

Suppose now that $\X = \X(\bSigma)$ is the toric Deligne--Mumford
stack associated to a stacky fan $\bSigma=(N, \Sigma, \rho)$, and that
$S$ is a finite set equipped with a map $S \to N_\Sigma = \{c\in N :
\bar{c} \in |\Sigma|\}$.  By composing it with a natural projection
$N_\Sigma \to \Box(\bSigma)$ we obtain a map $S\to \Box(\bSigma)$.  We
now identify $\LL^{S\vee}$ with $\Pic^S\X(\bSigma)$.
  
Let $m = |S|$ and let $s_1,\dots,s_m\in N_\Sigma$ be the images of 
elements of $S$ in $N_\Sigma$.  
The fan sequence and the $S$-extended fan sequence fit 
into the following commutative diagram: 
\begin{equation}
\label{eq:fan_sequences}
\begin{aligned}
  \xymatrix{
    0 \ar[r] & \LL \ar[r]^{\iota} \ar[d] &  \LL^S \ar[r] \ar[d]  & \ZZ^m \ar@{=}[d] \\
    0 \ar[r] & \ZZ^n \ar[r]  \ar[d]_\rho & \ZZ^{n+m} \ar[r] \ar[d]_{\rho^S} & \ZZ^m \ar[r] \ar[d] & 0 \\
    0 \ar[r] &  N \ar@{=}[r] & N \ar[r] & 0 
  }
\end{aligned}
\end{equation} 
with exact rows and columns.  We give a splitting of the first row
over the rational numbers.  Define
  \[
\mu \colon \QQ^m \to \LL^S \otimes \QQ  
  \]
by sending the $j$-th standard basis vector to  
\[
e_{j+n} - \sum_{i\in \sigma(j)} s_{ji} e_i  \in \LL^S\otimes \QQ 
\subset \QQ^{n+m} 
\]
where $\sigma(j)$ is the minimal cone containing $\bar{s}_j$ and 
the positive numbers $s_{ji}$ are determined by  
$\sum_{i\in \sigma(j)} s_{ji} \bar{\rho}_i = \bar{s}_j$. 
The map $\mu$ defines a splitting of the first row of 
\eqref{eq:fan_sequences} 
over $\QQ$: 
\begin{equation} 
\label{eq:LS_splitting} 
\LL^S \otimes \QQ   \cong (\LL\otimes \QQ)  \oplus \QQ^m 
\end{equation} 
Let $r_j$ be the order of the image of $s_j\in N_\Sigma$ 
in $N/\sum_{i\in \sigma(j)} \ZZ \rho_i$.  
Then we have $r_j s_{ji} \in \ZZ$. 
Therefore the dual of $\mu$ gives 
\[
\mu^* \colon 
\LL^{S\vee} \longrightarrow (\LL^S)^* \longrightarrow 
\bigoplus_{j=1}^m r_j^{-1} \ZZ. 
\]
One can check that the map $\mu^*$ together with the canonical map 
$\iota^* \colon \LL^\vee \to \LL^{S\vee}$ fits into 
the exact sequence: 
\[
\xymatrix{
  0 \ar[r] & \LL^{S\vee} \ar[rr]^-{(\iota^*, \mu^*)} && \LL^{\vee} \oplus {\bigoplus_{j=1}^m} r_j^{-1} \ZZ 
  \ar[rr]^-{(\res, \can)} && {\bigoplus_{j=1}^m} r_j^{-1} \ZZ/ \ZZ \ar[r] & 0
}
\]
where $\res$ maps an element of $\LL^\vee \cong \Pic(\X)$ to 
the ages of the corresponding line bundle 
at the box elements given by $s_1,\dots,s_m$ 
and $\can$ is the canonical projection.  
Thus we obtain:
\begin{prop}  
We have $\Pic^S \X(\bSigma) \cong \LL^{S\vee}$. 
\end{prop} 

We have $(\Pic \X(\bSigma))^*\otimes \RR \cong \LL \otimes \RR$ and
$(\Pic^S \X(\bSigma))^*\otimes \RR \cong \LL^S\otimes \RR$.  The
standard decomposition \eqref{eq:PicSdual_Picdual} matches with the
splitting \eqref{eq:LS_splitting}.  The Mori cone and the $S$-extended
Mori cone are described, as subsets of $\LL\otimes \RR$ and
$\LL^S\otimes \RR$, as follows:
\begin{align*} 
\NE(\X(\bSigma)) & = \sum_{\sigma \in \Sigma} C_\sigma^\vee \\
\NE^S(\X(\bSigma)) & = \iota(\NE(\X(\bSigma))) + \mu((\RR_{\ge 0})^m) \\
&\cong \NE(\X(\bSigma)) \times (\RR_{\ge 0})^m 
\quad \text{under \eqref{eq:LS_splitting}.}
\end{align*} 
Here $C_\sigma^\vee\subset \LL\otimes \RR$ is the dual cone of
\[
C_\sigma =\sum_{\substack{i : 1\le i\le n, \\i\notin \sigma}} \RR_{\ge 0} D_i 
\subset \LL^\vee \otimes \RR. 
\]
Our semi-projectivity assumption implies that the Mori cone $\NE(\X(\bSigma))$ is strictly convex.

\begin{defn}
Recall that $\LL^S \subset \ZZ^{n+m}$, where $m= |S|$. 
For a cone $\sigma \in \Sigma$, 
denote by
$\Lambda_\sigma^S\subset \LL^S \otimes \QQ$ the subset consisting of
elements 
  \[
  \lambda=\sum_{i=1}^{n+m}\lambda_i e_i
  \]
  such that $\lambda_{n+j}\in \ZZ, 1\leq j\leq m$, and $\lambda_i\in
  \ZZ$ if $i\notin \sigma$ and $i \leq n$. Set
  $\Lambda^S:=\bigcup_{\sigma\in \Sigma} \Lambda^S_{\sigma}$.
\end{defn}

\begin{defn} \label{defn:reduction_function}
The \emph{reduction function} is
  \begin{equation*}
    \begin{aligned}
      v^S \colon \Lambda^S & \longrightarrow  \Box(\bSigma) \\
      \lambda & \longmapsto \sum_{i=1}^n
\lceil\lambda_i\rceil \rho_i+\sum_{j=1}^m\lceil \lambda_{n+j}\rceil s_j
    \end{aligned}
  \end{equation*} 
This sends an element of $\Lambda^S_\sigma$ to $\Box(\sigma)$ as 
we have $\overline{v^S(\lambda)} = \sum_{i=1}^n \langle -\lambda_i \rangle 
\bar{\rho}_i \in \sigma$ for $\lambda \in \Lambda^S_\sigma$. 
Note that $v^S(\lambda)_i = \langle -\lambda_i \rangle$. 
For a box element $b\in \Box(\bSigma)$, we set:  
\begin{align*} 
\Lambda^S_b &:= \{\lambda \in \Lambda^S: v^S(\lambda) = b\} 
\intertext{and define:} 
\Lambda E^S & := \Lambda^S \cap \NE^S(\X(\bSigma)) \\ 
\Lambda E^S_b & := \Lambda^S_b \cap \NE^S(\X(\bSigma)) 
\end{align*}
We have $\Lambda^S_b\subset \Lambda^S_\sigma$ 
if $b\in \Box(\sigma)$. 
\end{defn}

\begin{rmk}
Elements of $\Lambda E^S_b$ can be interpreted as the $S$-extended
degrees of certain orbifold stable maps, as follows.  Let $f \colon
(\C,x_1,\dots,x_k,x_\infty) \to \X$ be an orbifold stable map such
that $f|_{x_\infty}$ gives the box element $\hat{b}\in \Box(\bSigma)$
and the rest of the markings $x_1,\dots,x_k$ are $S$-decorated,
i.e.~each $x_i$ is assigned an element of $S$ that maps to the box
element $f|_{x_i}$.  Then $f$ is naturally an $(S\sqcup
\{\hat{b}\})$-decorated stable map and has the
$(S\sqcup\{\hat{b}\})$-extended degree of the form $(\kappa,1)\in
\LL^{S\sqcup \{\hat{b}\}} \subset \ZZ^{n+m} \times \ZZ$ for some
$\kappa \in \ZZ^{n+m}$.  On the other hand, we have a ``wrong-way map"
$\epsilon \colon \LL^{S \sqcup \{\hat{b}\}} \otimes \QQ \to
\LL^{S}\otimes \QQ$ defined by the commutative diagram
\[
\xymatrix{
  \LL^{S\sqcup \{\hat{b}\}} \otimes \QQ \ar[r]^-{\cong} \ar[d]_-\epsilon &
  (\LL\otimes \QQ) \oplus \QQ^{m+1} \ar[d]^-{\text{\rm projection}} \\
  \LL^{S}\otimes \QQ \ar[r]^-{\cong} &  (\LL\otimes \QQ) \oplus \QQ^m
}
\]
where the horizontal arrows are the splitting 
given in \eqref{eq:LS_splitting}. 
The map $\epsilon$ induces a bijection 
\begin{align*} 
 \epsilon \colon & \left\{ (\kappa,1) \in \LL^{S\sqcup \{\hat{b}\}} : 
\kappa \in \ZZ^{n+m} \right\}   \cong 
\Lambda^S_b \\
& \text{where} \quad 
\epsilon(\kappa,1)_i = \begin{cases} 
\kappa_i + \hat{b}_i & 1 \le i\le n \\ 
\kappa_i & n+1 \le i\le n+m 
\end{cases} 
\end{align*}
and sends the $(S\sqcup \{\hat{b}\})$-extended degree 
\[
\deg^{S \sqcup\{\hat{b}\}}(f) =(\kappa,1) 
\in \LL^{S\sqcup\{\hat{b}\}} \cap 
\NE^{S\sqcup \{\hat{b}\}}(\X(\bSigma))
\]
to the $S$-extended degree $\deg^S(f) \in \Lambda E^S_{b}$.  Here we
regard $f$ as being $S$-decorated by forgetting the last marking
$x_{\infty}$; we need to generalize
Definition~\ref{def:extended_degree} by allowing orbifold stable maps
with domain curves having unmarked stacky points (and in this case
$\deg^S(f)(L,\varphi)$ is not necessarily an integer).
\end{rmk}

\begin{lem}
\label{lem:addition_map} 
Let $\sigma, \sigma'$ be top-dimensional cones of $\Sigma$ such that
$\sigma|\sigma'$.  Let $b\in \Box(\sigma)$ and let $b' \in
\Box(\sigma')$.  Recall the set $\Lambda E_{\sigma,b}^{\sigma,b'}
\subset \LL \otimes \QQ$ from Definition~\ref{defn:fixed_curve_deg}.
Addition in $\LL^S\otimes \QQ$ induces a map $\Lambda E_{\sigma,
  b}^{\sigma',b'} \times \Lambda^S_{b'} \to \Lambda_{b}^S$. Moreover,
for fixed $d\in \Lambda E_{\sigma,b}^{\sigma',b'}$, the map $\lambda'
\mapsto \lambda' +d$ induces a bijection $\Lambda^S_{b'} \cong
\Lambda^S_b$.
\end{lem}
\begin{proof} 
  Take $\lambda' \in \Lambda^S_{b'}$ and $l(c,\sigma,j) \in \Lambda
  E_{\sigma,b}^{\sigma',b'}$.  Here $j$ is the index defined in
  Notation~\ref{notation:what_is_j} and $c$ is a positive number such
  that $\langle c \rangle = \hat{b}_j$ and that \eqref{eq:def_b'}
  holds for some $q'\in \ZZ_{\ge 0}$.  We need to show that $\lambda
  := \lambda' + l(c,\sigma,j) \in \Lambda^S_b$.  It suffices to show
  that $v^S(\lambda) = b$.  First we show $\lambda \in
  \Lambda^S_\sigma$.  As described in
  Definition~\ref{defn:what_is_l_c_j}, $l(c,\sigma,j)$ is given by the
  relation of the form
\begin{equation} 
\label{eq:relation_lcj} 
c \bar{\rho}_j + 
c' \bar{\rho}_{j'} + \sum_{i\in \sigma \cap \sigma'} c_i \bar{\rho}_i 
=0
\end{equation} 
Thus the $i$th component of $\lambda\in \LL^S\otimes \QQ 
\subset \QQ^{n+m}$ is given by 
\[
\lambda_i = 
\begin{cases} 
\lambda'_j + c & i=j \\ 
\lambda'_{j'} + c' & i= j' \\ 
\lambda'_i + c_i & i \in \sigma \cap \sigma' \\ 
\lambda'_i & \text{otherwise} 
\end{cases} 
\]
To show $\lambda\in \Lambda_\sigma^S$, it suffices 
to see that $\lambda'_{j'} + c' \in \ZZ$. 
Since $v^S(\lambda') = b'$, we have $\langle - \lambda'_{j'} \rangle = b'_{j'}$. 
On the other hand, \eqref{eq:def_b'} together with the relation  
\eqref{eq:relation_lcj} shows that 
\[
\lfloor c' \rfloor = q' \qquad \text{and} \qquad  \langle c' \rangle = b'_{j'}
\]
This proves $\lambda'_{j'} + c' \in \ZZ$ and hence 
$\lambda \in \Lambda_\sigma^S$. 
Now we show $v^S(\lambda) = b$. 
We already know that $v^S(\lambda)$ lies in 
$\Box(\sigma)$. On the other hand, 
\begin{align*} 
v^S(\lambda) & = \sum_{i\notin \sigma \cup \sigma'} 
\lceil \lambda'_i \rceil \rho_i + 
\sum_{i=1}^m \lceil \lambda'_{n+i} \rceil s_{n+i} 
+ \sum_{i\in \sigma\cap \sigma'} \lceil \lambda'_i + c_i \rceil \rho_i   
+ \lceil \lambda'_j + c \rceil \rho_j 
+ \lceil \lambda'_{j'} + c' \rceil \rho_{j'} \\
& = b' + \left(\lceil \lambda'_j + c \rceil  
- \lceil \lambda'_j \rceil \right)\rho_j 
+ \left(\lceil \lambda'_{j'} + c' \rceil - \lceil \lambda'_{j'} \rceil \right) \rho_{j'} 
+ \sum_{i\in \sigma \cap \sigma'} 
\left(\lceil \lambda'_i + c_i \rceil -\lceil \lambda'_i \rceil 
\right)\rho_i \\ 
&  \equiv b' + \lceil c \rceil \rho_j + 
q' \rho_{j'} 
\mod \sum_{i\in \sigma\cap \sigma'} \ZZ \rho_i 
\end{align*} 
where we used $\lambda'_j \in \ZZ$ and $ \lceil \lambda'_{j'} + c'
\rceil - \lceil \lambda'_{j'} \rceil = \lambda'_{j'} + c' - \lceil
\lambda'_{j'} \rceil = c' - \langle -\lambda'_{j'}\rangle = c' -
b'_{j'} = q'$.  The last expression is congruent to $b$ modulo
$\sum_{i\in \sigma} \ZZ \rho_i$ by \eqref{eq:def_b'}.  Therefore
$v^S(\lambda) = b$ as claimed.  For the converse, if $\lambda\in
\Lambda_b^S$, one can argue similarly to show that $\lambda
-l(c,\sigma,j)$ lies in $\Lambda_{b'}^S$.
\end{proof} 

\section{Toric Mirror Theorem}
\label{sec:toric_mir_thm}

In this Section we state the main result of this paper,
Theorem~\ref{I_is_on_the_cone}. 

\begin{notation}
  Let $\sigma \in\Sigma$ be a top-dimensional cone. We write
  $u_k(\sigma)$ for the character of $\TT$ given by the restriction of
  the line bundle $u_k$ to the $\TT$-fixed point
  $\X(\bSigma)_\sigma$.
\end{notation}

\begin{notation} 
Let $S$ be a finite set equipped with a map $S \to N_\Sigma$ 
and set $m = |S|$. 
For $\lambda \in  \LL^S \otimes\QQ$, we write 
\[
\lambda = (d, k)  \qquad d\in \LL\otimes \QQ, 
\quad k \in \QQ^m 
\]
under the splitting 
$\LL^S \otimes \QQ \cong (\LL \otimes \QQ) \oplus \QQ^{m}$ 
in \eqref{eq:LS_splitting}. 
If $\lambda \in \Lambda E^S\subset \LL^S \otimes \QQ$, 
we have $k\in (\ZZ_{\ge 0})^m$ and $d \in \NE(\X(\bSigma)) 
\cap H_2(X(\bSigma),\ZZ)$. 
In this case we write 
\[
\tilde{Q}^\lambda = Q^d x^k = Q^d x_1^{k_1} \cdots x_m^{k_m} 
\in \Lambda_{\rm nov}^\TT [\![x]\!] 
\]
where $\Lambda_{\rm nov}^\TT$ is the $\TT$-equivariant 
Novikov ring \eqref{eq:Tequiv_Novikov} and 
$x =(x_1,\dots,x_m)$ are variables. 
We call $\tilde{Q} =(Q,x)$ the \emph{$S$-extended Novikov variables}. 
\end{notation} 

\begin{defn}
  Let $\bSigma=(N,\Sigma,\rho)$ be a stacky fan, 
and let $S$ be a finite set equipped with a map 
$S \to N_\Sigma$. Set $m= |S|$ and regard 
$\LL^S\otimes \QQ$ as a subspace of $\QQ^{n+m}$. 
The {\em $S$-extended $\TT$-equivariant $I$-function} 
of $\X(\bSigma)$ is:
\begin{equation}
\label{eq:what_is_I}
I_{\X(\bSigma)}^S(\tilde{Q},z):=
ze^{\sum_{i=1}^n u_i t_i/z}
\sum_{b\in \Box(\bSigma)}
\sum_{\lambda\in \Lambda E^S_b} 
\tilde{Q}^\lambda e^{\lambda t} 
\left(
\prod_{i=1}^{n+m}
\frac{\prod_{\<a\>=\< \lambda_i \>, a \leq 0}(u_i+a z)}
{\prod_{\<a \>=\< \lambda_i \>, a\leq \lambda_i} (u_i+a z)} 
\right) 
y^b 
\end{equation}
\end{defn}
Some explanations are in order:
\begin{enumerate}
\item the summation range $\Lambda E^S_b\subset \LL^S \otimes \QQ$ 
was introduced in Definition~\ref{defn:reduction_function}. 
\item for each $\lambda \in \Lambda E^S_b$, we write 
$\lambda_i$ for the $i$th component of $\lambda$ 
as an element of $\QQ^{n+m}$. We have $\<\lambda_i\> = \hat{b}_i$ 
for $1\le i\le n$ and $\<\lambda_i\>=0$ for $n+1\le i\le n+m$. 
\item $u_i:=0$ if $n+1 \le i\le n+m$. For $i=1,\ldots,n$, $u_i$ is the
  $\TT$-equivariant first Chern class of the line bundle discussed in
  Section~\ref{sec:equiv_obj}.
\item $y^b$ is the identity class supported on the twisted sector
  $I\X(\bSigma)_b$ associated to $b\in \Box(\bSigma)$; see
  Section~\ref{sec:CR_coh}.
\item $t=(t_1,\dots,t_n)$ are variables, and $e^{\lambda t}:=\prod_{i=1}^n
  e^{(D_i\cdot d) t_i}$. 
\end{enumerate}
$I^S_{\X(\bSigma)}(\tilde{Q},z)$ is a formal power series in $Q$, $x$,
$t$ with coefficients in the localized equivariant Chen--Ruan
cohomology $H_{\CR,\TT}^\bullet(\X) \otimes_{R_\TT} S_{\TT\times
  \CC^\times}$, i.e.
\[
I^S_{\X(\bSigma)}(\tilde{Q},z) 
\in H_{\CR,\TT}^\bullet(\X) \otimes_{R_\TT} 
S_{\TT \times \CC^\times}[\![\NE(X) \cap H_2(X,\ZZ)]\!][\![x,t]\!] 
\]
where $S_{\TT\times \CC^\times} = \Frac(H_{\TT\times
  \CC^\times}^\bullet({\rm pt}))$ and $z$ is identified with the
$\CC^\times$-equivariant parameter; see
Remark~\ref{rmk:rational_Givental}.

\begin{defn}
  If $\bSigma=(N,\Sigma,\rho)$ and $S$ are as above and the
  coarse moduli space of $\X(\bSigma)$ is projective, then we
  define the {\em $S$-extended (non-equivariant) $I$-function} of
  $\X(\bSigma)$ by the same equation \eqref{eq:what_is_I}, but
  with $u_i$, $1 \leq i \leq n$, replaced by the non-equivariant first
  Chern class $D_i$. 
\end{defn}

\begin{rmk} 
\label{rmk:lambda_range}
One can replace the summation range $\Lambda E^S_b$ in the formula
\eqref{eq:what_is_I} with $\Lambda^S_b$ without changing the
$I$-function.  This is because the summand for $\lambda \in
\Lambda^S_b$ contains a factor $(\prod_{i: \lambda_i\in \ZZ_{<0}} u_i
) y^b$ which vanishes unless $\{\bar{\rho}_i : \lambda_i\in \ZZ_{<0}
\text{ or } \lambda_i \notin \ZZ\}$ spans a cone; in particular the
summand for $\lambda$ automatically vanishes unless $\lambda$ lies in
$\NE^S(\X(\bSigma))$.
\end{rmk}

The main result of this paper is:

\begin{thm}[(Toric Mirror Theorem)]
\label{I_is_on_the_cone}
Let $\bSigma=(N,\Sigma,\rho)$ be a stacky fan giving rise to a smooth
toric Deligne--Mumford stack $\X(\bSigma)$ with semi-projective coarse
moduli space, and let $S$ be a finite set equipped with a map $S \to
N_\Sigma$.  The $S$-extended $\TT$-equivariant $I$-function
$I^S_{\X(\bSigma)}(\tilde{Q},-z)$ is a $\Lambda_{\rm
  nov}^\TT[\![x,t]\!]$-valued point of the Lagrangian cone
$\sL_{\X(\bSigma)}$ for the $\TT$-equivariant Gromov--Witten theory of
$\X(\bSigma)$.
\end{thm}

\begin{cor}
  \label{cor:non-equivariant}
  Suppose that $\bSigma=(N,\Sigma,\rho)$ and $S$ are as in
  Theorem~\ref{I_is_on_the_cone}, and that the coarse moduli space of
  $\X(\bSigma)$ is projective.  Then the $S$-extended
  non-equivariant $I$-function of $\X(\bSigma)$ is a 
  $\Lambda_{\rm nov}[\![x,t]\!]$-valued point of 
  the Lagrangian cone $\sL_{\X(\bSigma)}$ for the
  non-equivariant Gromov--Witten theory of $\X(\bSigma)$.
\end{cor}
\begin{proof}
  Since the coarse moduli space of $\X(\bSigma)$ is
  projective, the non-equivariant Chen--Ruan cohomology, $S$-extended
  non-equivariant $I$-function of $\X(\bSigma)$, and
  non-equivariant Gromov--Witten theory of $\X(\bSigma)$ are
  well-defined. Pass to the non-equivariant limit in
  Theorem~\ref{I_is_on_the_cone}.
\end{proof}

\begin{rmk}
  Theorem~\ref{I_is_on_the_cone} and
  Corollary~\ref{cor:non-equivariant} take a particularly simple form
  when the pair $(\X(\bSigma), S)$ is {\em weak Fano}. Roughly
  speaking, in this case the $S$-extended $I$-function
  $I^S_{\X(\bSigma)}$ coincides with (a suitable restriction
  of) the $J$-function of $\X(\bSigma)$. See \cite[Section
  4.1]{ir} for more details.
\end{rmk}

\begin{rmk}
  The non-extended $I$-function (i.e.~the $S$-extended $I$-function
  with $S = \varnothing$) typically only determines the restriction of
  the $J$-function to the ``very small parameter space'' $H^2(\X;\CC)
  \subset H^2_\CR(\X;\CC)$. Taking $S$ to be non-trivial in
  Theorem~\ref{I_is_on_the_cone} and
  Corollary~\ref{cor:non-equivariant}, however, in practice often
  allows one to determine the $J$-function along twisted sectors too.
  But it is convenient to take $S$ not to be too large---not equal to
  the whole of $\Box(\X)$, for example---as otherwise we may lose
  control over the asymptotics of the $I$-function.  We
  will elaborate on these points elsewhere.
\end{rmk}

\begin{rmk} 
  The $S$-extended $I$-function arises from Givental's heuristic
  argument \cite{Givental:homological} applied to the polynomial loop
  spaces (toric map spaces) associated to the $S$-extended quotient
  construction \eqref{eq:extended_quotient} of $\X(\bSigma)$.  See
  \cite{Givental:toric, Vlassopoulos, Iritani:eqf, cclt} for
  closely-related discussions.
\end{rmk}

The remainder of this paper contains a proof of
Theorem~\ref{I_is_on_the_cone}. We first give a criterion, in
Theorem~\ref{thm:char_of_cone}, that characterizes points on the
Lagrangian cone $\sL_{\X(\bSigma)}$.  We then show, in
Section~\ref{sec:pf_mir_thm}, that the $S$-extended $I$-function
$I^S_{\X(\bSigma)}$ satisfies the criterion in
Theorem~\ref{thm:char_of_cone}.

\section{Lagrangian Cones in the Toric Case}
\label{sec:toric_lag_cone}

Let $\X = \X(\bSigma)$ be the toric Deligne--Mumford stack
associated to a stacky fan $\bSigma=(N, \Sigma, \rho)$, as in
Section~\ref{sec:toric_basics}.  
In this Section we characterize those
points of $\sH$ which lie on Givental's Lagrangian cone $\sL_\X$ 
associated to $\TT$-equivariant Gromov--Witten theory of $\X(\bSigma)$ 
(see Section~\ref{sec:Givental_formalism}). 
Recall that the $\TT$-fixed points of $\X(\bSigma)$ are in bijection
with top-dimensional cones of $\Sigma$: given a top-dimensional cone
$\sigma\in \Sigma$, we have a fixed point
\[
\X(\bSigma/\sigma) \cong \X(\bSigma)_\sigma \subset
\X(\bSigma).
\]
Note that $\X(\bSigma/\sigma)\simeq BN(\sigma)$, where
$N(\sigma):= N/N_\sigma$ and $N_\sigma\subset N$ is the subgroup
generated by $\rho_i, i\in \sigma$.  

\begin{notation} 
For a top-dimensional cone $\sigma \in \Sigma$, we write 
$T_\sigma \X(\bSigma)$ for the tangent space at the $\TT$-fixed 
point $\X(\bSigma)_\sigma$. This is a $\TT$-equivariant 
vector bundle over $\X(\bSigma)_\sigma \cong B N(\sigma)$. 
\end{notation} 

\begin{notation}
Let $\sigma \in \Sigma$ be a top-dimensional cone.  
We write $\sH_\sigma$ for Givental's symplectic 
vector space associated to the $\TT$-fixed point $\X(\bSigma)_\sigma$.  
We also write $\sH_\sigma^\tw$ and $\sL_\sigma^\tw$ 
for the symplectic vector space and Lagrangian cone corresponding 
to the Gromov--Witten theory of $\X(\bSigma)_\sigma$, 
twisted by the vector bundle $T_\sigma \X(\bSigma)$ and 
the $\TT$-equivariant inverse Euler class $e_\TT^{-1}$. 
More precisely: 
\begin{align*} 
\sH_\sigma & :=  H_{\CR}(\X(\bSigma)_\sigma) 
\otimes_\CC S_\TT(\!(z^{-1})\!)[\![\NE(X) \cap H_2(X,\ZZ)]\!] \\ 
\sH_\sigma^\tw & := H_{\CR}(\X(\bSigma)_\sigma) 
\otimes_\CC S_\TT(\!(z)\!) [\![\NE(X)\cap H_2(X,\ZZ)]\!] 
\end{align*}
See Sections~\ref{sec:Givental_formalism} and~\ref{sec:twisted}.
Although there are no Novikov variables for the stacky point
$\X(\bSigma)_\sigma$, we define $\sH_\sigma$, $\sH_\sigma^\tw$ over
the Novikov ring of $\X(\bSigma)$ by extending scalars.
\end{notation}

\begin{notation}
  By the Atiyah--Bott localization theorem, we have the isomorphism
\begin{align}
    \label{eq:localization_isomorphism}
    H^\bullet_{\CR,\TT}(\X(\bSigma)) 
    \otimes_{R_\TT} S_\TT 
    & \simeq \bigoplus_{\sigma\in \Sigma: \text{ top-dimensional}}
    H^\bullet_{\CR}(\X(\bSigma)_\sigma) 
    \otimes_{\CC} S_\TT 
    \intertext{given by restricting to $\TT$-fixed points, and thus
      an isomorphism of vector spaces:}
    \sH & \simeq \bigoplus_{\sigma\in \Sigma:
      \text{ top-dimensional}} \sH_\sigma 
    \notag
\end{align} 
Under this isomorphism, the symplectic form on $\sH$ 
corresponds to the direct sum of $\eT^{-1}$-twisted symplectic forms
on $\bigoplus_\sigma \sH_\sigma$. 
For $\fb \in \sH$ and $\sigma \in \Sigma$ a top-dimensional cone, we
  write $\fb_\sigma\in \sH_\sigma$ for the component of $\fb$
  along $\sH_\sigma \subset \sH$.  
Thus $\fb_\sigma$ is the restriction of $\fb$ to the 
inertia stack $I\X(\bSigma)_\sigma$ of the $\TT$-fixed point
$\X(\bSigma)_\sigma$. 
We write $\fb_{(\sigma,b)}$ for the restriction of 
$\fb_\sigma$ to the component $I\X(\bSigma)_{\sigma,b}$ of 
$I\X(\bSigma)_\sigma$ corresponding to $b\in \Box(\sigma)$. 
The component $I\X(\bSigma)_{\sigma,b}$ is contained in 
both $I\X(\bSigma)_\sigma$ and $I\X(\bSigma)_b$. 
\end{notation}

\begin{defn}[(Recursion Coefficient)]
  \label{defn:recursion_coefficient}
  Let $\bSigma=(N,\Sigma,\rho)$ be a stacky fan, and let
  $\sigma$,~$\sigma' \in \Sigma$ satisfy $\sigma|\sigma'$.  Let $j$ be
  as in Notation~\ref{notation:what_is_j}.  Fix $b \in \Box(\sigma)$,
  and let $c$ be a positive rational number such that $\langle c
  \rangle = \hat{b}_j$ with $\hat{b} = \inv(b)$.  The \emph{recursion
    coefficient} associated to $(\sigma, \sigma', b, c)$ is the
  element of $S_\TT = \Frac(H_\TT^\bullet({\rm pt})) \cong
  \CC(\chi_1,\dots,\chi_d)$ given by:
  \[
  \RC(c)_{(\sigma,b)}^{(\sigma',b')} 
  := \frac{1}{c} 
  \left(\prod_{i\in \sigma : b_i=0}u_i(\sigma)\right)
  \frac{(\frac{c}{u_j(\sigma)})^{\lfloor c \rfloor}}{\lfloor c \rfloor !} 
  \frac{({-\frac{c}{u_j(\sigma)}})^{\lfloor c' \rfloor}}{\lfloor c' \rfloor !} 
  \prod_{i\in \sigma\cap \sigma'}
  \frac{\prod_{\<a\>=\hat{b}_i, a< 0}
    (u_i(\sigma)+u_j(\sigma) \frac{a}{-c})}
  {\prod_{\<a\>=\hat{b}_i, a\leq c_i}
    (u_i(\sigma)+u_j(\sigma)\frac{a}{-c})}
  \]   
  where $c', c_i$ are as in Definition~\ref{defn:what_is_l_c_j},
  i.e.~$c'= D_{j'}\cdot l(c,\sigma,j)$ and $c_i = D_i \cdot
  l(c,\sigma,j)$ for $i\in \sigma \cap \sigma'$.
\end{defn}

\begin{rmk}
The recursion coefficient $\RC(c)_{(\sigma,b)}^{(\sigma',b')}$
depends only on $\sigma$,~$\sigma'$,~$b$, and~$c$.  The box element
$b'\in \Box(\sigma')$ is determined by these data, 
via Remark~\ref{rmk:determined}.
\end{rmk}

\begin{thm}\label{thm:char_of_cone}
Let $\X = \X(\bSigma)$ be a smooth toric Deligne--Mumford stack
associated to a stacky fan $\bSigma=(N, \Sigma, \rho)$. 
Let $x =(x_1,\dots,x_m)$ be formal variables. 
Let $\fb$ be an element of $\sH[\![x]\!]$ such that 
$\fb|_{Q=x=0} = -1 z$. 
Then $\fb$ is a $\Lambda_{\rm nov}^\TT[\![x]\!]$-valued point of 
$\sL_\X$ if and only if the following three conditions hold: 

  \begin{enumerate}
  \item[(C1)] For each top-dimensional cone $\sigma \in \Sigma$ and
    each $b\in \Box(\sigma)$, the restriction $\fb_{(\sigma,b)}$ is a
    power series in $Q$ and $x$ such that each coefficient of this
    power series is an element of $S_{\TT\times \CC^\times} \cong
    \CC(\chi_1,\dots,\chi_d,z)$ and, as a function in $z$, it is
    regular except possibly for a pole at $z=0$, a pole at $z=\infty$, and
    simple poles at:
      \[
      \Big\{ 
      \textstyle
      \frac{{u_j}(\sigma)}{c} : 
      \text{$\exists \sigma' \in \Sigma$ 
such that $\sigma|\sigma'$ and $j \in \sigma \setminus \sigma'$, 
        $c>0$ is such that $\langle c \rangle =
        \hat{b}_j$} \Big\}
      \]
      Here we use Notation~\ref{notation:what_is_j}.
\item[(C2)] The residues of $\fb_{\sigma,b}$ at the simple poles satisfy the
      following recursion relations: given any $\sigma$,~$\sigma'\in
      \Sigma$ such that $\sigma|\sigma'$, 
$b\in \Box(\sigma)$ and $c>0$ with $\<c\> = \hat{b}_j$,  we have:
      \[
      \Res_{z= \frac{u_j(\sigma)}{c}} \fb_{(\sigma, b)}(z) \, dz=
      {-  Q^{l(c,\sigma, j)} } \,
      \RC(c)_{(\sigma, b)}^{(\sigma', b')} \,
      \fb_{(\sigma', b')}(z)\Bigr |_{z=\frac{u_j(\sigma)}{c}}. 
      \]
Here we use Notation~\ref{notation:what_is_j}, 
Definition~\ref{defn:what_is_l_c_j} and 
Definition~\ref{defn:recursion_coefficient}.  

  \item[(C3)] The Laurent expansion of 
$\fb_\sigma$ at $z=0$ is a $\Lambda_{\rm nov}^\TT[\![x]\!]$-valued 
point of $\sL_\sigma^\tw$.
  \end{enumerate}
\end{thm}

\begin{rmk} 
\label{rmk:regularity} 
Condition (C1) ensures that the right-hand side of the recursion
relation in (C2) is well-defined.  Note that the $\TT$-weights
$\{u_i(\sigma') :i \in \sigma'\}$ of the tangent space $T
_{\sigma'}\X(\bSigma)$ form a basis of $\Lie(\TT)^*$ and all simple
poles of $\fb_{\sigma',b'}(z)$ are contained in the cone $\sum_{i\in
  \sigma'} \RR_{\ge 0} u_i(\sigma')$ by (C1).  On the other hand, if
we take the representable morphism $f\colon \PP_{r_1,r_2} \to
\X(\bSigma)$ associated to $\sigma,c,j,b$ in
Proposition~\ref{prop:what_is_c}, then $u_j(\sigma)/c$ and
$u_{j'}(\sigma')/c'$ are the induced $\TT$-weights of the tangent
spaces at $0$ and $\infty$ of the \emph{coarse} domain curve
$|\PP_{r_1,r_2}|\cong \PP^1$ (see Definition~\ref{defn:what_is_l_c_j}
for $c'$).  Therefore we have $u_j(\sigma)/c = - u_{j'}(\sigma')/c'$
and it follows that $\fb_{\sigma',b'}(z)$ is regular at $z
=u_j(\sigma)/c = -u_{j'}(\sigma')/c'$.
\end{rmk} 
\begin{rmk} 
  Note that (C3) involves analytic continuation: (C1) implies that
  each coefficient of $Q^d x^k$ in $\fb_\sigma(z)$ is a rational
  function in $z, \chi_1,\dots,\chi_d$, and so it makes sense to take
  the Laurent expansion at $z=0$.
\end{rmk} 

\begin{proof}[Proof of Theorem~\ref{thm:char_of_cone}] 
  In outline: Brown's proof for toric bundles \cite[Theorem~2]{brown}
  works for toric Deligne--Mumford stacks too.  In detail: we argue as
  follows.

  Suppose first that $\fb$ is a $\Lambda^{\TT}_{\rm nov}[\![x]\!]$-valued 
point on  $\sL_\X$.  Then
  \begin{equation}
    \label{eq:fb_GW}
    \fb = {-1} z+ \bt(z)+
    \sum_{n=0}^\infty
    \sum_{d\in\NE(\X)}
    \sum_{\alpha}
    \frac{Q^d}{n!}
    \<\bt(\bpsi),\ldots,\bt(\bpsi),\frac{\phi_\alpha}{-z-\bpsi}\>_{0,n+1,d}^\TT
    \phi^\alpha
  \end{equation}
  for some $\bt(z) \in \sH_+[\![x]\!]$ with $\bt|_{Q=x=0} = 0$; 
  here once again we expand the 
  expression $\phi_\alpha/(-z-\bpsi)$ as a power series in
  $z^{-1}$. 
Under the isomorphism \eqref{eq:localization_isomorphism}, 
the identity class 
$1 \in H_{\CR,\TT}^\bullet(\X(\bSigma))$ and $\bt(z) \in \sH_+[\![x]\!]$ 
correspond to  
  \begin{align*}
\bigoplus_{\sigma\in \Sigma: \text{ top-dimensional}} 
1_\sigma && \text{and} && 
\bigoplus_{\sigma\in \Sigma:
      \text{ top-dimensional}} \bt_\sigma(z)
  \end{align*}
where $1_\sigma$ is the identity element in 
$H^\bullet_{\CR,\TT}(\X(\bSigma)_\sigma)$ 
and $\bt_\sigma(z) \in \sH_{\sigma,+}$, 
and we have
  \[
  \fb_\sigma = {-1_\sigma z} + \bt_\sigma(z) + 
  \iota_\sigma^\star \left[
    \sum_{n=0}^\infty
    \sum_{d\in\NE(\X)}
    \sum_{\alpha}
    \frac{Q^d}{n!}
    \<\bt(\bpsi),\ldots,\bt(\bpsi),\frac{\phi_\alpha}{-z-\bpsi}\>_{0,n+1,d}^\TT
    \phi^\alpha
  \right]
  \]
  where $\iota_\sigma \colon \X(\bSigma)_\sigma \to
  \X(\bSigma)$ is the inclusion of the $\TT$-fixed point. 
Furthermore
  \begin{equation}
    \label{eq:f_sigma_b}
    \fb_{(\sigma,b)} = -\delta_{b,0} z + \bt_{(\sigma,b)}(z) + 
    \sum_{n=0}^\infty
    \sum_{d\in \NE(\X)} \frac{Q^{d}}{n!} \left\langle
      \frac{1^{\sigma,b}}{-z-\bpsi},\bt,\ldots,\bt \right\rangle_{0,n+1,d}^\TT
\end{equation}
where $1^{\sigma,b} := |N(\sigma)| e_{\TT}(N_{\sigma,b}) 
1_{\sigma,\hat{b}}$, $N_{\sigma,b}$ is the normal bundle to
  $I\X(\bSigma)_{\sigma,b}$ in $I\X(\bSigma)_b$,
  and $1_{\sigma,\hat{b}}$ is the fundamental class of
  $I\X(\bSigma)_{\sigma,\hat{b}}$ with $\hat{b} = \inv(b)$. 

  We compute the sum in \eqref{eq:f_sigma_b} using localization in
  $\TT$-equivariant cohomology.  Chiu-Chu Melissa Liu has produced a detailed
  and beautifully-written introduction to localization in
  $\TT$-equivariant Gromov--Witten theory of toric stacks
  \cite{ccliu}; we follow her notation closely.  $\TT$-fixed
  strata in the moduli space $\Mbar_{0,n+1}(\X, d)$
  are indexed by decorated trees $\Gamma$, where: 
  \begin{itemize}
  \item Each vertex $v$ of $\Gamma$ is labelled by a
    top-dimensional cone $\sigma_v \in \Sigma$.
  \item Each edge $e$ of $\Gamma$ is labelled by a codimension-$1$
    cone $\tau_e \in \Sigma$ and a positive integer $d_e$.
  \item Each flag\footnote{A flag $(e,v)$ of $\Gamma$ is an
      edge-vertex pair such that $e$ is incident to $v$.}  $(e,v)$ of
    $\Gamma$ is labelled with an element $k_{(e,v)}$ of the isotropy
    group $G_v$ of the $\TT$-fixed point
    $\X(\bSigma)_{\sigma_v}$.
  \item There are markings $\{1,2,\ldots,n+1\}$ and a map $s \colon
    \{1,2,\ldots,n+1\} \to V(\Gamma)$ that assigns markings to
    vertices of $\Gamma$.
  \item The marking $j \in \{1,2,\ldots,n+1\}$ is labelled with an
    element $k_j \in G_v$, where $v = s(j)$.
  \item A number of compatibility conditions hold.
  \end{itemize}
  The compatibility conditions that we require are spelled out in
  detail in \cite[Definition~9.6]{ccliu}; they include, for example,
  the requirement if $(e,v)$ is a flag of $\Gamma$ then the
  $\TT$-fixed point determined by $\sigma_v$ is contained in the
  closure of the $1$-dimensional $\TT$-orbit determined by
  $\tau_e$. We denote the set of all decorated trees satisfying the
  compatibility conditions by $G_{0,n+1}(\X,d)$, so that $\TT$-fixed
  strata in $\Mbar_{0,n+1}(\X, d)$ are indexed by decorated trees
  $\Gamma \in G_{0,n+1}(\X,d)$.

  We rewrite equation \eqref{eq:f_sigma_b} as
  \begin{equation}
    \label{eq:f_sigma_b_as_graph_sum}
    \fb_{(\sigma,b)} = -\delta_{b,0} z + \bt_{(\sigma,b)}(z) + 
    \sum_{n=0}^\infty
    \sum_{d\in \NE(\X)} \frac{Q^d}{n!}
    \sum_{\Gamma \in G_{0,n+1}(\X,d)}
    \Contr_{\sigma,b}(\Gamma)
  \end{equation}
  where $\Contr_{\sigma,b}(\Gamma)$ denotes the contribution to the
  $\TT$-equivariant Gromov--Witten invariant
  \[
  \left\langle\frac{1^{\sigma,b}}{-z-\bpsi},\bt,\ldots,\bt
  \right\rangle_{0,n+1,d}^\TT
  \]
  from the $\TT$-fixed stratum $\mathcal{M}_\Gamma \subset
  \Mbar_{0,n+1}(\X, d)$ corresponding to $\Gamma$.
  We will need some notation for graphs.  For a decorated graph
  $\Gamma \in G_{0,n+1}(\X,d)$, we write:
  \begin{itemize}
  \item $V(\Gamma)$ for the set of vertices of $\Gamma$;
  \item $E(\Gamma)$ for the set of edges of $\Gamma$;
  \item $F(\Gamma)$ for the set of flags of $\Gamma$;
  \item $S_v$ for the set of markings assigned to the vertex $v \in
    V(\Gamma)$:
    \[
    S_v = \big\{ j \in \{1,2,\ldots,n+1\} : s(j) = v \big\}
    \]
  \item $E_v$ for the set of edges incident to the vertex $v \in
    V(\Gamma)$:
    \[
    E_v = \big\{e \in E(\Gamma) : (e,v) \in F(\Gamma)\big\}
    \]
  \item $\val(v) = |E_v| + |S_v|$ for the valence of the vertex $v \in
    V(\Gamma)$.
    \end{itemize}
    Liu \cite[Theorem~9.32]{ccliu} shows that the contribution from
    $\mathcal{M}_\Gamma$ to the Gromov--Witten invariant 
    \[
    \langle
    \gamma_1 \bpsi_1^{a_1},\ldots,\gamma_{n+1}
    \bpsi_{n+1}^{a_{n+1}} \rangle_{0,n+1,d}
    \]
    is:
    \begin{multline}
    \label{eq:Liu_formula}
    c_\Gamma 
    \prod_{e \in E(\Gamma)} 
    \frac{\eT(H^1(\cC_e,f_e^\star T\X)^\mov)}
    {\eT(H^0(\cC_e,f_e^\star T\X)^\mov)} 
    \prod_{(e,v) \in F(\Gamma)} \eT\big((T_{\sigma_v} \X)^{k_{(e,v)}}\big)
    \prod_{v \in V(\Gamma)} \left( \prod_{j : s(j) = v}
      \iota_{\sigma_v}^\star \gamma_j \right) \\
    \prod_{v \in V(\Gamma)} \int_{\left[ 
\Mbar^{\overrightarrow{b(v)}}_{0,\val(v)}(BG_v)\right]^w}
    \frac{\prod_{j \in S_v} \bpsi_j^{a_j}}
    {\prod_{e \in E_v} \big(w_{(e,v)} -
      \bpsi_{(e,v)}/r_{(e,v)}\big)} \cup
    \eT^{-1}\big((T_{\sigma_v}\X)_{0,\val(v),0}\big) 
  \end{multline}
  where:
  \begin{align*}
    & \text{$c_\Gamma = \frac{1}{|\Aut(\Gamma)|}
      \prod_{e \in E(\Gamma)} \frac{1}{(d_e |G_e|)}
      \prod_{(e,v) \in F(\Gamma)} \frac{|G_v|}{r_{(e,v)}}$} \\
    & \text{$G_e$ is the generic stabilizer of the one-dimensional 
    toric substack $\X(\bSigma/\tau_e)$} \\ 
    & \text{$f_e \colon \cC_e \to \X$ is the toric map 
    to the one-dimensional toric substack $\X(\bSigma/\tau_e)$
      determined} \\ 
    & \qquad \text{by the edge $e$ and the decorations 
    $\tau_e,d_e, \{k_{(e,v)}: \text{$v$ is a vertex incident to $e$}\}$} \\ 
    & \text{$H^i(\C_e,f_e^*T\X)^\mov$ denotes 
   the moving part with respect to the $\TT$-action} \\ 
    &\text{$w_{(e,v)} = \eT(T_{\mathfrak{y}(e,v)} \cC_e)$, where
      $\mathfrak{y}(e,v)$ is the marked point on $\cC_e$ determined by
      $(e,v)$} \\
    & \text{$r_{(e,v)}$ is the order of $k_{(e,v)} \in G_v$} \\
    & \text{$\overrightarrow{b(v)}$ is determined by the decorations 
    $k_j$,~$j\in S_v$, and $k_{(e,v)}$,~$e \in E_v$} \\
    & \text{$(T_{\sigma_v}\X)_{0,\val(v),0}$ is the twisting bundle 
    associated to the vector bundle $T_{\sigma_v}\X$ over the $\TT$-fixed} \\ 
   & \qquad \text{point $\X(\bSigma/\sigma_v)$;  
      see Section~\ref{sec:twisted}}\\ 
    & \text{ $\Mbar^{\overrightarrow{b(v)}}_{0,\val(v)}(BG_v)$ 
       is taken to be a point if $\val(v)=1$ or $\val(v)=2$ } 
\end{align*}
The integrals over $\Mbar^{\overrightarrow{b(v)}}_{0,\val(v)}(BG_v)$
here in the unstable cases $\val(v)=1$ and $\val(v)=2$ are defined as
in \cite[(9.12)--(9.14)]{ccliu}.  The twisting bundles
$(T_{\sigma_v}\X)_{0,\val(v),0}$ in the unstable cases $\val(v)=1$ and
$\val(v)=2$ are defined to be $(T_{\sigma_v}\X)^{k_{(e,v)}}$; see the
end of \cite[\S 9.3.4]{ccliu}.

Consider now the graph sum in \eqref{eq:f_sigma_b_as_graph_sum}. Each
graph $\Gamma$ in the sum contains a distinguished vertex $v$ that
carries the first marked point.  We may assume both that $\sigma_v =
\sigma$ and that the label $k_1$ of the first marking is equal to
$\hat{b}$, as otherwise the contribution of $\Gamma$ is zero.  There
are two possibilities: 
\begin{itemize}
\item[(A)] $v$ is $2$-valent;
\item[(B)] $v$ has valence at least $3$.
\end{itemize}
In the first case we say that $\Gamma$ has type A, and in the second
case we say that $\Gamma$ has type B; see Figures~\ref{fig:typeA}
and~\ref{fig:typeB}.  As we will see below, the contributions from
type~A graphs have simple poles at points of the form $u_j(\sigma)/c$
as described in the statement of the Theorem, and the contributions
from type~B graphs are polynomials in $z^{-1}$.  The condition (C1)
then follows.

\begin{figure}[htbp]
\centering
\begin{picture}(400,100) 

\put(30,50){\vector(-2,1){20}}
\put(30,50){\makebox(0,0){$\bullet$}}
\put(30,50){\line(1,0){90}} 
\shade{
\put(150,50){\circle{60}}
}
\put(120,50){\makebox(0,0){$\bullet$}}

\put(10,67){\makebox(0,0){$1$}} 
\put(30,42){\makebox(0,0){$v$}} 
\put(75,42){\makebox(0,0){$e$}} 
\put(115,42){\makebox(0,0){$v'$}} 
\put(70,5){\makebox(0,0){type A graph $\Gamma$}} 

\put(220,70){\makebox(0,0){\scriptsize deleting}}
\put(220,60){\makebox(0,0){\scriptsize the edge $e$}}
\put(220,45){\makebox(0,0){$\Longrightarrow$}} 

\put(300,50){\vector(-1,0){20}}
\shade{ 
\put(330,50){\circle{60}}
} 
\put(300,50){\makebox(0,0){$\bullet$}} 

\put(280,57){\makebox(0,0){$1$}}
\put(295,42){\makebox(0,0){$v'$}} 
\put(300,5){\makebox(0,0){$\Gamma'$}} 
\end{picture} 
\caption{} 
\label{fig:typeA}
\end{figure}

Consider a graph $\Gamma$ of type A.
Let $e \in E(\Gamma)$ be the edge incident to $v$. 
Then $\Gamma$ is obtained from another decorated graph 
$\Gamma'$ by adding the decorated vertex $v$ and 
the decorated edge $e$. See Figure~\ref{fig:typeA}. 
Let $v'$ be the other vertex incident to $e$. 
The graph $\Gamma'$ is assigned the first marking at $v'$ 
instead of the edge $e$. 
The map $f_e \colon \cC_e \to \X$
determined by the edge $e$ has $\cC_e
\simeq \PP_{r_{(e,v)},r_{(e,v')}}$, $f_e(0) = \X_{\sigma_v}$, and
$f_e(\infty) = \X_{\sigma_{v'}}$; the restriction 
$f_e|_0 \colon B\mu_{r_{(e,v)}} 
\to \X_{\sigma_v}$ gives $\hat{b} \in
\Box(\sigma_v)$.  
Let $c \in \QQ$ and $b' \in \Box(\sigma_{v'})$ be
the rational number and box element determined by applying
Proposition~\ref{prop:what_is_c} and Remark~\ref{rmk:determined} 
to $f_e$, and write $\sigma' = \sigma_{v'}$. 
Since $\bpsi_1 = -r_{(e,v)} w_{(e,v)}$, we obtain: 
\begin{multline*}
  \Contr_{\sigma,b}(\Gamma)  =
  \frac{c_\Gamma}{c_{\Gamma'}}
  \frac{\eT(H^1(\cC_e,f_e^\star T\X)^\mov)}
  {\eT(H^0(\cC_e,f_e^\star T\X)^\mov)} 
  \eT\big((T_{\sigma} \X)^{k_{(e,v)}}\big) 
  \eT\big((T_{\sigma'} \X)^{k_{(e,v')}}\big) \\
  \int_{\left[\Mbar^{(\hat{b},b)}_{0,2}(BG_v)\right]^w}
  \frac{|N(\sigma_v)| e_{\TT}(N_{\sigma_v,b})}
  {\big({-z} + r_{(e,v)}w_{(e,v)}\big)} 
  \frac{1}
  {\big(w_{(e,v)} - \bpsi_2/r_{(e,v)}\big)} 
  \cup
  \eT^{-1}\big((T_{\sigma_v} \X)^{k_{(e,v)}}\big) \\
  \times
\frac{r_{(e,v')}}{|N(\sigma')| \eT(N_{\sigma',b'})}
 \Contr_{\sigma',b'}(\Gamma')\big|_{z={-r_{(e,v')}} w_{(e,v')}}
\end{multline*}
Calculating the ratio $c_\Gamma/c_{\Gamma'}$ and evaluating the
integral over $\Mbar^{(\hat{b},b)}_{0,2}(BG_v)$ using
\cite[(9.14)]{ccliu} yields:
\[
\Contr_{\sigma,b}(\Gamma)
=   
\frac{|G_v|}{d_e |G_e|}
\frac{\eT(H^1(\cC_e,f_e^\star T\X)^\mov)}
{\eT(H^0(\cC_e,f_e^\star T\X)^\mov)} 
\frac{e_{\TT}(N_{\sigma,b})}
{\big({-z} + r_{(e,v)} w_{(e,v)}\big)} 
\left[ \Contr_{\sigma',b'}(\Gamma')\right]_{z={-r_{(e,v')}} w_{(e,v')}}
\]
Liu has computed the ratio of Euler classes here
\cite[Lemma~9.25]{ccliu}, and in our notation this gives:
\[
\Contr_{\sigma,b}(\Gamma)
=   
\frac{\RC(c)_{(\sigma, b)}^{(\sigma', b')}}
{( -z + u_j(\sigma)/c) } 
\left[ \Contr_{\sigma',b'}(\Gamma')\right]_{z=u_j(\sigma)/c}
\]
where we used $r_{(e,v)} w_{(e,v)} = u_j(\sigma)/c = - r_{(e,v')}
w_{(e,v')}$.  (See Remark~\ref{rmk:Liu_comparison} below for a
detailed comparison between Liu's notation and ours.)  Note that the
degree of the map $f_e \colon \cC_e \to \X$ is $l(c,\sigma,j)$; see
Definition~\ref{defn:what_is_l_c_j}.  Note also that if we hold the
decorated vertex $v$ and the decorated edge $e$ constant (or in other
words, if we hold the map $f_e \colon \cC_e \to \X$ constant) then the
sum of $\Contr(\Gamma')_{\sigma',b'}$ over all compatible trees
$\Gamma'$ is exactly\footnote{We elide a subtle detail here: the
  unstable terms $[-\delta_{b',0} z + \bt_{(\sigma',b')}(z) ]_{z=
    u_j(\sigma)/c}$ in $[\fb_{(\sigma',b')}]_{z=u_j(\sigma)/c}$ arise
  from the graphs $\Gamma$ in the sum such that $\Gamma'$ is unstable,
  with only one vertex and one or two markings attached to it.} the
graph sum that defines $\fb_{(\sigma',b')}$.  Thus the contribution to
\begin{equation*}
  \fb_{(\sigma,b)} = -\delta_{b,0} z + \bt_{(\sigma,b)}(z) + 
  \sum_{n=0}^\infty
  \sum_{d\in \NE(\X)} 
  \sum_{\Gamma \in G_{0,n+1}(\X,d)}
  \frac{Q^d}{n!}
  \Contr_{\sigma,b}(\Gamma)
\end{equation*}
from all graphs $\Gamma$ of type A is:
\begin{equation}
  \label{eq:type_A}
  \sum_{\sigma' : \sigma|\sigma'}
  \sum_{\substack{
      c \in \QQ : c > 0, \\
      \langle c \rangle = \hat{b}_j}}
  Q^{l(c,\sigma,j)}
  \frac{\RC(c)_{(\sigma, b)}^{(\sigma', b')}}
  {(-z + u_j(\sigma)/c)} 
  \left[ \fb_{(\sigma',b')}\right]_{z=u_j(\sigma)/c}
\end{equation}
This proves (C2). 

\begin{figure}[htbp]
\centering
\begin{picture}(400,110)
\put(40,50){\vector(-1,0){20}} 
\put(40,50){\makebox(0,0){$\bullet$}} 
\put(40,50){\line(3,2){60}}
\put(40,50){\line(1,0){60}} 
\put(40,50){\vector(1,-1){30}}  
\put(115,90){\shade{\circle{30}}} 
\put(115,50){\shade{\circle{30}}} 
\put(100,90){\makebox(0,0){$\bullet$}} 
\put(100,50){\makebox(0,0){$\bullet$}} 

\put(40,40){\makebox(0,0){$v$}} 
\put(70,80){\makebox(0,0){$e_1$}} 
\put(70,55){\makebox(0,0){$e_2$}} 
\put(90,95){\makebox(0,0){$v_1'$}} 
\put(92,58){\makebox(0,0){$v_2'$}} 
\put(18,43){\makebox(0,0){$1$}} 
\put(75,20){\makebox(0,0){$k$}} 
\put(80,-5){\makebox(0,0){type B graph $\Gamma$}} 

\put(190,70){\makebox(0,0){\scriptsize decomposing $\Gamma$}}  
\put(190,60){\makebox(0,0){\scriptsize into type A subgraphs}}
\put(190,45){\makebox(0,0){$\Longrightarrow$}} 

\put(270,90){\vector(-1,0){20}} 
\put(270,90){\makebox(0,0){$\bullet$}}
\put(270,90){\line(1,0){60}} 
\put(345,90){\shade{\circle{30}}}
\put(330,90){\makebox(0,0){$\bullet$}} 
\put(270,50){\vector(-1,0){20}} 
\put(270,50){\makebox(0,0){$\bullet$}}
\put(270,50){\line(1,0){60}} 
\put(345,50){\shade{\circle{30}}}
\put(330,50){\makebox(0,0){$\bullet$}} 
\put(270,20){\vector(-1,0){20}}
\put(270,20){\makebox(0,0){$\bullet$}} 
\put(270,20){\vector(2,-1){30}} 

\put(380,90){\makebox(0,0){$\Gamma_2$}} 
\put(380,50){\makebox(0,0){$\Gamma_3$}} 
\put(330,5){\makebox(0,0){$\Gamma_4$}} 

\put(270,99){\makebox(0,0){$v$}} 
\put(270,59){\makebox(0,0){$v$}} 
\put(270,29){\makebox(0,0){$v$}} 
\put(300,95){\makebox(0,0){$e_2$}} 
\put(300,55){\makebox(0,0){$e_3$}} 
\put(323,99){\makebox(0,0){$v_2'$}} 
\put(323,59){\makebox(0,0){$v_3'$}} 
\put(245,90){\makebox(0,0){$1$}} 
\put(245,50){\makebox(0,0){$1$}} 
\put(245,20){\makebox(0,0){$1$}} 
\put(305,5){\makebox(0,0){$k$}} 
\end{picture} 
\caption{} 
\label{fig:typeB} 
\end{figure} 

Write:
\begin{align*}
  & \btau_{(\sigma,b)}(z) := \bt_{(\sigma,b)}(z) + (\text{the quantity in
    equation~\ref{eq:type_A}}) \\
  \intertext{and:}
  & \btau_\sigma(z) := \sum_{b \in \Box(\sigma)} \btau_{(\sigma,b)}(z)
  1_b
\end{align*}
We have that:
\begin{align}
  \fb_\sigma & = \sum_{b \in \Box(\sigma)} \fb_{(\sigma,b)} 1_b
  \notag \\
  \label{eq:just_type_B_left}
  & = {-1_\sigma z} + \btau_\sigma(z) +   
  \sum_{n=0}^\infty
  \sum_{d\in \NE(\X)} 
  \sum_{b \in \Box(\sigma)}
  \sum_{\substack{
      \Gamma \in G_{0,n+1}(\X,d): \\
      \text{$\Gamma$ is of type B}}}
    \frac{Q^d}{n!}
  \Contr_{\sigma,b}(\Gamma) 1_{\sigma,b}
\end{align}
Consider the contribution to \eqref{eq:just_type_B_left} given by the
sum over decorated graphs $\Gamma$ of type B such that the
distinguished vertex $v$ has valence $l$ and that the label $k_1$ of the
distinguished vertex is equal to $\hat{b} \in \Box(\sigma)$.  Each
such graph $\Gamma$ gives contributions of the form
\eqref{eq:Liu_formula}.  We evaluate these contributions by
integrating over all the factors
$\Mbar_{0,\val(v')}^{\vec{b'}}(BG_{v'})$ except that associated with
the distinguished vertex $v$, obtaining an expression of the form\footnote
{Here $\Aut_{\Gamma_2,\dots,\Gamma_l}$ is the subgroup of 
the symmetric group $\mathfrak{S}_{l-1}$ which leaves the 
$(l-1)$-tuple $(\Gamma_2,\dots,\Gamma_{l})$ of decorated 
graphs invariant.}:
\[
\frac{1}{|\Aut_{\Gamma_2,\dots,\Gamma_l}|}
\int_{\Mbar_{0,l}^{\hat{b},b^2,\ldots,b^l}(BG_v)}
\frac{1^{\sigma,b}}{-z-\bpsi_1} \cup
h_2(\bt,\bpsi_2) \cup \cdots \cup h_l(\bt,\bpsi_l) 
\cup \eT^{-1}\big((T_{\sigma_v}\X)_{0,l,0}\big) \, 1_{\sigma,b}
\]
for some elements $b^2,\ldots,b^l \in \Box(\sigma)$ and some
polynomials $h_i(\bt,\bpsi_i)$ in $t_0$,~$t_1$,\ldots, $Q$, and $\bpsi_i$. 
Suppose that $\Gamma$ is obtained from type A subgraphs 
$\Gamma_2,\dots,\Gamma_l$ by joining them at the 
distinguished vertex $v$, as in Figure~\ref{fig:typeB}. 
If $\Gamma_i$ consists of one vertex with two markings (such as
$\Gamma_4$ in Figure~\ref{fig:typeB}) then $h_i(\bt,\bpsi_i) =
\bt_{(\sigma,b^i)}(\bpsi_i)$.  Otherwise $h_i(\bt,\bpsi_i)$ records a
more complicated contribution determined by the subgraph $\Gamma_i$;
we have:
\[
h_i(\bt,\bpsi_i) = Q^{d_i}
\Contr_{\sigma,b^i}(\Gamma_i)|_{z=\bpsi_i}
\]
where $d_i$ is the total degree of the contribution from $\Gamma_i$. 
Now fix $v$ and all
other parts of $\Gamma$ except the subtree $\Gamma_i$, and sum over
all possible subtrees $\Gamma_i$: the total contribution of the
$h_i(\bt,\bpsi_i)$s is \eqref{eq:type_A} with $b=b^i$ and $z=\bpsi_i$.  
Thus the contribution to \eqref{eq:just_type_B_left} given by the sum over
decorated graphs $\Gamma$ of type B such that the distinguished vertex
has valence $l$ and that the label $k_1$ of the distinguished vertex
is equal to $\hat{b} \in \Box(\sigma)$ is:
\[
\frac{1}{(l-1)!}
\int_{\Mbar_{0,l}(BG_v)} 
\frac{1^{\sigma,b}}{-z-\bpsi_1} \cup
\btau_{\sigma}(\bpsi_2) \cup \cdots \cup 
\btau_{\sigma}(\bpsi_l) \cup
\eT^{-1}\big((T_{\sigma}\X)_{0,l,0}\big)
\]
These are twisted Gromov--Witten invariants of the $\TT$-fixed point
$\X(\bSigma)_\sigma$.  
Summarizing, we see that \eqref{eq:just_type_B_left} becomes:
\begin{equation*}
  \fb_\sigma = {-1_\sigma z} + \btau_\sigma(z) +   
  \sum_{l=3}^\infty
  \sum_{b \in \Box(\sigma)}
  \frac{1}{(l-1)!}
  \<\frac{1^{\sigma,b}}{-z-\bpsi},
  \btau_\sigma(\psi),\ldots,\btau_\sigma(\psi)
  \>_{0,l,0}^\tw
  1_{\sigma,b}
\end{equation*}
The superscript `$\tw$' indicates that these are Gromov--Witten
invariants of $\X(\bSigma)_\sigma$ twisted by the vector
bundle $T_\sigma \X(\bSigma)$ and the $\TT$-equivariant
inverse Euler class $e_\TT^{-1}$.  Using 
\eqref{eq:very_big_J_function_twisted}, 
we see that the Laurent expansion at $z=0$ of $\fb_\sigma$ 
lies in $\sL_\sigma^\tw$.  Thus we have proved (C3).

Conversely, suppose that $\fb \in \sH[\![x]\!]$ satisfies $\fb|_{Q=x=0}=-1 z$ 
and conditions (C1),~(C2), and~(C3) in the statement of the Theorem.  
Conditions~(C1) and~(C2) together imply that:
\begin{align}
  \label{eq:fb_hypothetically_not_GW}
  \begin{split} 
  \fb_{\sigma} &= {-1_\sigma}z + \bt_{\sigma} + 
  \sum_{b\in \Box(\sigma)} 1_b 
 \sum_{\sigma' : \sigma|\sigma'} 
   \sum_{\substack{
      c \in \QQ : c > 0, \\
      \langle c \rangle = \hat{b}_j}}
  Q^{l(c,\sigma,j)}
  \frac{\RC(c)_{(\sigma, b)}^{(\sigma', b')}}
  {\big( -z + u_j(\sigma)/c\big)} 
  \fb_{(\sigma',b')}\big|_{z=u_j(\sigma)/c} \\
  & \qquad \qquad + O(z^{-1})
  \end{split} 
\end{align}
for some $\bt_{\sigma} \in \sH_{\sigma,+}[\![x]\!]$ with 
$\bt_{\sigma}|_{Q=t=0}= 0$.  
The remainder $O(z^{-1})$ is a formal power series 
in $Q$ and $x$ with coefficients in $z^{-1} S_\TT[z^{-1}]$. 
Let $\bt_\GW \in
\sH_+$ be the unique element such that its restriction to
$I\X(\bSigma)_{\sigma}$ is $\bt_{\sigma}$, and let
$\fb_\GW$ be the element of $\sL_\X$ defined by \eqref{eq:fb_GW} with
$\bt = \bt_\GW$.  Then, in view of the first part of the proof, we
have that $\fb_\GW$ and $\fb$ both satisfy conditions (C1--C3), and
both give rise to the same values $\bt_{\sigma}$ in
\eqref{eq:fb_hypothetically_not_GW}.  It therefore suffices to show
that $\fb$ can be reconstructed uniquely from the collection
\begin{equation}
\label{eq:collection}
\big(\bt_{\sigma} : \text{$\sigma \in \Sigma$ is a top-dimensional
  cone}\big)
\end{equation} 
using condition~(C3).

We argue by induction on the degree with respect to $Q$ and $x$. 
Pick a K\"{a}hler class $\omega$ of $\X(\bSigma)$ 
and assign the degree $\int_d \omega + \sum_{i=1}^m k_i$ 
to the monomial $Q^d x_1^{k_1}\cdots x_m^{k_m}$.  
Let $\kappa_0 = \int_{d_0} \omega>0$ be the minimal possible degree 
of a non-constant stable map.  
Suppose that $\fb$ is uniquely determined from the collection 
\eqref{eq:collection} up to order $\kappa$. 
We shall show that $\fb$ is determined up to order $\kappa+ \kappa_0$. 
We know by \eqref{eq:fb_hypothetically_not_GW} that 
$\fb_{\sigma}$ is determined up to order $\kappa + \kappa_0$ 
except for the term $O(z^{-1})$. 
On the other hand, under the Laurent expansion at $z=0$, 
all the quantities in the first line of \eqref{eq:fb_hypothetically_not_GW} 
lie in $\sH_{\sigma,+}^\tw$. 
Therefore, in view of \eqref{eq:very_big_J_function_twisted}, 
the term $O(z^{-1})$ is uniquely determined up to order $\kappa+\kappa_0$ 
from the quantities in the first line by 
condition (C3), i.e.~that the Laurent expansion at $z=0$ of 
$\fb_\sigma$ lies in $\sL_\sigma^\tw$. 
This completes the induction and the 
proof of Theorem~\ref{thm:char_of_cone}. 
\end{proof} 

\begin{rmk} 
\label{rmk:Liu_comparison} 
For the convenience of the reader, we compare Liu's notation
\cite[Lemma 9.25]{ccliu} with ours.  Consider a decorated graph
$\Gamma$ occurring in the proof above, and an edge $e\in E(\Gamma)$
with incident vertices $v$,~$v'\in V(\Gamma)$.  The edge $e$
corresponds to a toric representable morphism $f = f_e \colon
\PP_{r_1,r_2} \to \X(\bSigma)$ given by $\sigma,\sigma', c,b$ in
Proposition~\ref{prop:what_is_c}, where $\sigma = \sigma_v$ and
$\sigma' = \sigma_{v'}$.  Let $j$ and $j'$ be the indices in
Notation~\ref{notation:what_is_j}.  Recall (from
Definition~\ref{defn:what_is_l_c_j}) that the degree $l(c,\sigma,j)
\in \LL\otimes\QQ$ of the map $f$ is given by the relation:
\[
c \bar{\rho}_j + c' \bar{\rho}_{j'} + \sum_{i\in \sigma\cap\sigma'} c_i 
\bar{\rho}_i = 0 
\] 
Set $\tau = \tau_e = \sigma \cap \sigma'$.  Then Liu's
quantities\footnote{The definition of $a_i$ in \cite[\S 8.6]{ccliu}
  contains a typo; it should be the integral over the
  \emph{rigidification} of $\X(\bSigma)_\tau$.}
$\mathbf{w}(\tau,\sigma)$, $\mathbf{w}(\tau,\sigma')$,
$\mathbf{w}(\tau_i,\sigma)$, $\mathbf{w}(\tau_i',\sigma')$,
$r(\tau,\sigma)$, $\mathbf{u} = r(\tau,\sigma)
\mathbf{w}(\tau,\sigma)$, $d=d_e$, $a_i$ $\epsilon_i$, $r_{(e,v)}$,
$r_{(e,v')}$, $w_{(e,v)}$, $w_{(e,v')}$ are given in our notation as:
\begin{align*} 
\mathbf{w}(\tau,\sigma) & = u_j(\sigma) \quad 
\mathbf{w}(\tau,\sigma') = u_{j'}(\sigma') \\ 
\mathbf{w}(\tau_i,\sigma) &= u_i (\sigma) \quad 
\mathbf{w}(\tau_i',\sigma') = u_i(\sigma') \qquad 
\text{for $i \in \sigma\cap \sigma'$} \\
r(\tau,\sigma) & := \text{the order of the stabilizer at $\X(\bSigma)_\sigma$ 
of the rigidification $\X(\bSigma)_{\tau}^{\rm rig}$ of $\X(\bSigma)_\tau$} 
\\ 
& = \frac{|G_v|}{|G_e|} = 
\frac{|N(\sigma)|}{|N(\tau)_{\tor}|} = \text{the norm of the 
image of $\rho_j$ in $\overline{N(\tau)}\cong \ZZ$} \\
\mathbf{u} & = r(\tau,\sigma) u_j(\sigma) 
\\
d & := \text{the degree of the map 
$(f\colon |\PP_{r_1,r_2}| \to |\X(\bSigma)_\tau|)$ between 
the coarse curves ($\cong \PP^1$)} \\
& = r(\tau,\sigma) c  = r(\tau,\sigma') c' \\ 
a_i & := \int_{\X(\bSigma)_{\tau}^{\rm rig}} u_i 
= \frac{c_i}{d} = \frac{c_i}{c} r(\tau,\sigma)^{-1} \\ 
\epsilon_i & = \hat{b}_i \qquad \text{for} \quad i\in \sigma\cap\sigma' \\
r_{(e,v)} & = r_1, \quad 
r_{(e,v')} = r_2, \\
r_{(e,v)} w_{(e,v)} & = u_j(\sigma)/c, \quad 
r_{(e,v')} w_{(e,v')} = - r_{(e,v)} w_{(e,v)} = 
u_{j'}(\sigma')/{c'}   
\end{align*} 
where we set $N(\tau) = N/\sum_{i\in \tau} \ZZ \rho_i$, 
$\overline{N(\tau)} = N(\tau)/N(\tau)_{\rm tor}$ 
and $N(\tau)_{\tor}$ is the torsion part 
of $N(\tau)$.   
Our recursion coefficient $\RC(c)_{(\sigma,b)}^{(\sigma',b')}$ 
coincides with $\frac{1}{c} \eT(N_{\sigma,b}) \mathbf{h}(e)$ 
where $\mathbf{h}(e)$ is in \cite[(9.26)]{ccliu}. 
\end{rmk}

\section{Proof of Theorem~\ref{I_is_on_the_cone}}
\label{sec:pf_mir_thm}

In this Section we complete the proof of Theorem~\ref{I_is_on_the_cone}, 
by showing that the $S$-extended $I$-function 
\[
I^S_{\X(\bSigma)}(\tilde{Q},-z)
\] 
satisfies the conditions in Theorem~\ref{thm:char_of_cone}. 
This amounts to proving 
Propositions~\ref{prop:I_condition_1},~\ref{prop:recursion_for_I},
and~\ref{prop:restriction_for_I} below.
Note that the sign of $z$ should be flipped when we consider 
the $I$-function. 

For a top-dimensional cone $\sigma\in \Sigma$ and $b\in \Box(\sigma)$, 
we write $I^S_{\sigma}(\tilde{Q},z)$ and $I^S_{(\sigma,b)}(\tilde{Q},z)$ 
for the restrictions of $I^S_{\X(\bSigma)}(\tilde{Q},z)$ to the 
inertia stack $I\X(\bSigma)_\sigma$ of the $\TT$-fixed point $\X(\bSigma)_\sigma$ 
and the component $I\X(\bSigma)_{\sigma,b}$ of $I\X(\bSigma)_\sigma$ respectively. 

\begin{prop}
  \label{prop:I_condition_1}
  The extended $I$-function satisfies condition (C1) in
  Theorem~\ref{thm:char_of_cone}.  In other words, for each
  top-dimensional cone $\sigma\in \Sigma$ and $b\in \Box(\sigma)$,
  $I^S_{(\sigma,b)}(\tilde{Q}, z)$ is a power series in the extended
  Novikov variables $\tilde{Q}$ and $t$ such that each coefficient of
  this power series lies in $S_{\TT\times \CC^\times} =
  \CC(\chi_1,\dots,\chi_d,z)$ and, as a function of $z$, it is regular
  except possibly for a pole at $z=0$, a pole at $z=\infty$, and
  simple poles at:
  \[
  \Big\{ 
  \textstyle
  \frac{{-u_j}(\sigma)}{c} : 
  \text{$\exists \sigma' \in \Sigma$ 
    such that $\sigma|\sigma'$ and $j \in \sigma \setminus \sigma'$, 
    $c>0$ is such that $\langle c \rangle =
    \hat{b}_j$} \Big\}
  \]
  Here we use Notation~\ref{notation:what_is_j}.
\end{prop}

\begin{prop}\label{prop:recursion_for_I}
  The extended $I$-function satisfies condition (C2) in
  Theorem~\ref{thm:char_of_cone}.  In other words, for any
  $\sigma$,~$\sigma'\in \Sigma$ such that $\sigma|\sigma'$, we have:
  \[
  \Res_{z=-\frac{u_j(\sigma)}{c}} I^S_{(\sigma,b)}(\tilde{Q}, z) \, dz 
  = Q^{l(c,\sigma, j)} \,
  \RC(c)_{(\sigma, b)}^{(\sigma', b')} \,
  I^S_{(\sigma',b')}(\tilde{Q}, z)\Big|_{z=-\frac{u_j(\sigma)}{c}}
  \]
\end{prop}

\begin{prop}\label{prop:restriction_for_I}
The extended $I$-function satisfies condition (C3) in
Theorem~\ref{thm:char_of_cone}.  
In other words, if $\sigma\in \Sigma$ is a top-dimensional cone, 
then the Laurent expansion at $z=0$ of $I^S_\sigma(\tilde{Q},-z)$ 
is a $\Lambda_{\rm nov}^\TT[\![x,t]\!]$-valued point of $\sL_\sigma^\tw$. 
\end{prop}

\subsection{Poles of the Extended $I$-Function} 
\label{sec:I_condition_1} 
In this subsection we prove Proposition~\ref{prop:I_condition_1}. 
Let $\sigma$ be a top-dimensional cone and take 
$b\in \Box(\sigma)$. 
The restriction $I^S_{(\sigma,b)}$ of the $I$-function 
to the fixed point $I\X(\bSigma)_{\sigma,b}$ 
takes the form: 
\begin{equation} 
\label{eq:I_sigma_b}
I^S_{(\sigma,b)}(\tilde{Q},z) = z e^{\sum_{i=1}^n u_i(\sigma)t_i/z} 
\sum_{\lambda \in \Lambda E^S_b} 
\tilde{Q}^\lambda e^{\lambda t} 
\left( 
\prod_{i\notin \sigma} \frac{\prod_{a\le 0, \<a\>=0} a z}
{\prod_{a\le \lambda_i, \<a\>=0} a z} 
\right) 
\left( 
\prod_{i\in \sigma} \frac{\prod_{a\le 0,\<a\>=\<\lambda_i\>} 
(u_i(\sigma) + az)}{\prod_{a\le \lambda_i, \<a\> = \<\lambda_i\>} 
(u_i(\sigma) + az)} 
\right) 
\end{equation} 
where the index $i$ ranges over $\{1,\dots,n+m\}$ and we regard
$\sigma\subset \{1,\dots,n\}$ as a subset of $\{1,\dots,n+m\}$.  We
also used $u_i(\sigma)=0$ for $i\notin \sigma$.  For $\lambda\in
\Lambda^S_b$, we have that $\lambda_i\in \ZZ$ for all $i\notin \sigma$
because $\<\lambda_i\> = \hat{b}_i$ and $b\in \Box(\sigma)$.  Note
also that one may assume that $\lambda_i\in \ZZ_{\ge 0}$ for $i\notin
\sigma$ in the above sum, as otherwise the contribution is zero.  We
see that $I^S_{(\sigma,b)}$ has poles possibly at $z=0$ and $z=\infty$
and simple poles at
\begin{align*}
  -u_i(\sigma)/a && \text{with} &&
  \text{$0<a\le \lambda_i$, $\<a\>=\<\lambda_i\>
    =\hat{b}_i$, $i\in \sigma$}
\end{align*}
for $\lambda \in \Lambda E^S_b$ contributing to the sum.  It suffices
to see that, if $\lambda_{i_0}>0$ for some $i_0\in\sigma$, then there
exists a top-dimensional cone $\sigma'$ such that $\sigma|\sigma'$ and
$i_0\in \sigma \setminus \sigma'$, i.e. $i_0=j$ in Notation
\ref{notation:what_is_j}.  We have
\begin{equation}
\label{eq:lambda_i_relation}
\sum_{i\in \sigma}(-\lambda_i) \bar{\rho}_i = \sum_{\substack{i : 1
    \le i\le n \\ i\notin \sigma}} 
\lambda_i \bar{\rho}_i + \sum_{i=1}^m \lambda_{n+i} \bar{s}_i 
\end{equation}
where $s_1,\dots,s_m$ are the images of elements of $S$ in $N_\Sigma$.
As we remarked above, we may assume that $\lambda_i\in \ZZ_{\ge 0}$
for $i\notin \sigma$ and hence the right-hand side belongs to the
support $|\Sigma|$ of the fan.  Therefore $\sum_{i\in
  \sigma}(-\lambda_i)\bar{\rho}_i \in |\Sigma|$.  Because $|\Sigma|$
is convex, the positivity of $\lambda_{i_0}$ implies that there exists
a top-dimensional cone $\sigma'\in \Sigma$ such that $\sigma|\sigma'$
and $i_0\in\sigma \setminus \sigma'$.
Proposition~\ref{prop:I_condition_1} is proved.

\subsection{Recursion for the Extended $I$-Function}

In this subsection we prove Proposition~\ref{prop:recursion_for_I}.
Let $\sigma \in \Sigma$ be a top-dimensional cone and let $b\in
\Box(\sigma)$.  Fix another top-dimensional cone $\sigma'$ with
$\sigma|\sigma'$ and a positive rational number $c$ such that $\<c\> =
\hat{b}_j$, where $j$ is the index in
Notation~\ref{notation:what_is_j}.  We examine the residue of
$I^S_{(\sigma,b)}$ at $z=-u_j(\sigma)/c$.  Write
\[
\Boxshape_{\lambda,i,\sigma}(z) = 
\frac{\prod_{\<a\> = \<\lambda_i\>, a\le 0} 
(u_i(\sigma) + a z)}{\prod_{\<a\> = \<\lambda_i\>, a\le \lambda_i}
(u_i(\sigma) + a z)}
\]
for $\lambda \in \Lambda^S$ and $1\le i \le n+m$. 
The residue of \eqref{eq:I_sigma_b} at $z=-u_j(\sigma)/c$ 
is given by: 
\begin{equation}
\label{eq:residue} 
\left( - \tfrac{u_j(\sigma)}{c}\right) 
e^{\frac{\sum_{i=1}^n u_i(\sigma) t_i}{-u_j(\sigma)/c}} 
\frac{1}{c}  
\sum_{\substack{\lambda\in \Lambda^S_b \\ 
\lambda_j \ge c}} \tilde{Q}^\lambda 
e^{\lambda t} 
\frac{\prod_{i: i\neq j} \Boxshape_{\lambda,i,\sigma}(-\tfrac{u_j(\sigma)}{c}) }
{\prod_{\substack{0<a\le \lambda_j, \<a\>=\<\lambda_j\>\\  a\neq c}}
(u_j(\sigma)  - a \frac{u_j(\sigma)}{c})}
\end{equation} 
Recall from Remark~\ref{rmk:lambda_range} that the summation 
range can be taken to be $\Lambda^S_b$ instead of $\Lambda E^S_b$. 
Let $l(c,\sigma,j) \in \Lambda E_{\sigma,b}^{\sigma',b'}
\subset \LL\otimes \QQ$ be the degree from Definition~\ref{defn:what_is_l_c_j}. 
We now consider the change of variables 
\[
\lambda = \lambda' + l(c,\sigma,j) 
\]
and replace the sum over $\lambda\in \Lambda_b^S$ with the sum over 
$\lambda'\in \Lambda_{b'}^S$ using Lemma~\ref{lem:addition_map}. 
We write $c_i$ for the components of $l(c,\sigma,j)\in \LL\otimes \QQ 
\subset \LL^S \otimes \QQ$ as an element of $\QQ^{n+m}$. 
Using the notation in Definition~\ref{defn:what_is_l_c_j}, 
we have $c_i = D_i \cdot l(c,\sigma,j)$ for $1\le i\le n$, 
$c_j=c$, $c_{j'} = c'$ and $c_i=0$ for $n+1 \le i\le n+m$. 

\begin{lem} 
  \label{lem:useful_things}
Let $\lambda, \lambda'$ be as above. We have:
\begin{align} 
\label{eq:weight_difference}
&  u_i(\sigma) = u_i(\sigma') + \frac{c_i}{c} u_j(\sigma) \\ 
\label{eq:exponents}
&  \frac{\sum_{i=1}^n u_i(\sigma) t_i}{-u_j(\sigma)/c} + \lambda t = 
\frac{\sum_{i=1}^n u_i(\sigma') t_i}{-u_j(\sigma)/c} + \lambda' t \\ 
\label{eq:recursion_coeff_main}
& \Boxshape_{\lambda,i,\sigma}(-\tfrac{u_j(\sigma)}{c}) = 
\Boxshape_{\lambda',i,\sigma'}(-\tfrac{u_j(\sigma)}{c}) 
\frac{\prod_{a \le 0, \<a\>=\langle \lambda_i\rangle}
(u_i(\sigma)-\frac{a}{c}u_j(\sigma)) }
{\prod_{a\le c_i, \<a\> = \langle \lambda_i\rangle}
(u_i(\sigma)-\frac{a}{c} u_j(\sigma))} && 
\text{for $i\neq j$}\\ 
\label{eq:jth_factor} 
& \prod_{\substack{0<a\le \lambda_j, \<a\>=\<\lambda_j\>\\  a\neq c}}
(u_j(\sigma)  - a \tfrac{u_j(\sigma)}{c}) 
 = \prod_{\substack{-c<a\le \lambda'_j, a\in \ZZ \\ a\neq 0}} 
(- a \tfrac{u_j(\sigma)}{c})  
\end{align} 
\end{lem} 
\begin{proof} 
  The formulas \eqref{eq:exponents} and
  \eqref{eq:recursion_coeff_main} follow easily from
  \eqref{eq:weight_difference}; the formula \eqref{eq:jth_factor} is
  obvious if we notice $\langle c\rangle = \hat{b}_j = \langle
  \lambda_j \rangle$ and $\lambda_j = \lambda'_j + c$.  It suffices to
  show \eqref{eq:weight_difference}.  The equality
  \eqref{eq:weight_difference} is obvious for $n+1\le i\le m$, so we
  restrict to the case where $1\le i\le n$.  Consider the
  representable morphism $f\colon \PP_{r_1,r_2} \to \X(\bSigma)$ given
  by $(\sigma,\sigma',b,c)$ via Proposition~\ref{prop:what_is_c}.  By
  the localization formula, we obtain
\begin{align*} 
c_i & = D_i \cdot l(c,\sigma,j) = \int_{\PP_{r_1,r_2}} f^* D_i 
=\int_{\PP_{r_1,r_2}}^\TT f^*u_i 
= \frac{u_i(\sigma)}{u_j(\sigma)/c} + \frac{u_i(\sigma')}{-u_j(\sigma)/c} 
\end{align*} 
where we use the fact that $u_j(\sigma)/c$ and $-u_j(\sigma)/c$ are
the induced $\TT$-weights at $0$ and $\infty$ of the coarse domain
curve $|\PP_{r_1,r_2}|$: see Remark~\ref{rmk:regularity}.  The Lemma
follows.
\end{proof} 

Applying Lemma~\ref{lem:useful_things}, we see that \eqref{eq:residue}
equals
\begin{multline*} 
\left( - \tfrac{u_j(\sigma)}{c}\right) 
e^{\frac{\sum_{i=1}^n u_i(\sigma') t_i}{-u_j(\sigma)/c}} 
\frac{1}{c}  
\sum_{\substack{\lambda'\in \Lambda^S_{b'} \\ 
\lambda'_j \ge 0}} 
\tilde{Q}^{\lambda'} e^{\lambda' t} 
\frac{\prod_{i: i\neq j} \Boxshape_{\lambda',i,\sigma'}(-\tfrac{u_j(\sigma)}{c}) }
{\prod_{\substack{-c<a\le \lambda'_j, a\in \ZZ \\ a\neq 0}} 
(-a\frac{u_j(\sigma)}{c})}
\prod_{i: i\neq j} 
\frac{\prod_{a \le 0, \<a\>=\langle \lambda_i\rangle}
(u_i(\sigma)-\frac{a}{c}u_j(\sigma)) }
{\prod_{a\le c_i, \<a\> = \langle \lambda_i\rangle}
(u_i(\sigma)-\frac{a}{c} u_j(\sigma))}  \\
= 
\frac{1}{c} 
\frac{1}{\prod_{0<a<c, a\in \ZZ} \left( a \frac{u_j(\sigma)}{c}\right)}  
\prod_{i \in \sigma'} 
\frac{\prod_{a \le 0, \<a\>=\langle \lambda_i\rangle}
(u_i(\sigma)-\frac{a}{c}u_j(\sigma)) }
{\prod_{a\le c_i, \<a\> = \langle \lambda_i\rangle}
(u_i(\sigma)-\frac{a}{c} u_j(\sigma))} 
I^S_{(\sigma',b')}(\tilde{Q},z)\Big|_{z= -u_j(\sigma)/c}
\end{multline*}
multiplied by $Q^{l(c,\sigma,j)}$. 
It is now straightforward to check that the last expression 
coincides with 
\[
\RC(c)_{(\sigma,b)}^{(\sigma',b')} I^S_{(\sigma',b')}(\tilde{Q},z)
\Big|_{z= -u_j(\sigma)/c} 
\]
(Here we used $u_{j'}(\sigma) = 0$ and $\lambda_{j'} \in
\ZZ$.)\phantom{.}  Proposition~\ref{prop:recursion_for_I} is proved.

\subsection{Restriction of the Extended $I$-Function to Fixed Points}

In this subsection we prove Proposition~\ref{prop:restriction_for_I}. 
Let $\sigma\in \Sigma$ be a top-dimensional cone. 
By \eqref{eq:I_sigma_b} and the discussion in 
Section \ref{sec:I_condition_1}, 
the restriction $I^S_\sigma(\tilde{Q},-z)$ of the $S$-extended $I$-function to 
the $\TT$-fixed point $\X(\bSigma)_\sigma$ is: 
\begin{equation*}
-ze^{-\sum_{i=1}^n u_i(\sigma) t_i/z} 
\sum_{\substack{\lambda \in \Lambda_\sigma^S \\ 
\lambda_i \in \ZZ_{\ge 0} \text{ if } i\notin \sigma}}
\frac{\tilde{Q}^\lambda e^{\lambda t} }
{\prod_{i\notin \sigma}\lambda_i!(-z)^{\lambda_i}}  
\left( \prod_{j \in \sigma} 
\frac{\prod_{\< a \>=\langle \lambda_j \rangle, a\leq 0}(u_j (\sigma) -a z)}
{\prod_{\<a\>=\langle \lambda_j \rangle, a \leq \lambda_j}(u_j (\sigma)- a z)}
\right) 
1_{v^S(\lambda)}
\end{equation*} 
where $1_{v^S(\lambda)}\in H^\bullet_{\CR}(\X(\bSigma)_\sigma)$ 
is the identity class supported on the twisted sector corresponding 
to $v^S(\lambda)\in \Box(\sigma)$. 
We want to show that this lies on the Lagrangian cone $\sL^\tw_\sigma$.  
We claim that it suffices to show that $I^S_\sigma(\tilde{Q},-z)|_{t=0}$ 
lies on $\sL^\tw_\sigma$. 
By the String Equation, $\sL^\tw_\sigma$ is invariant under multiplication 
by $e^{-\sum_{i=1}^n u_i(\sigma) t_i/z}$ and thus 
we can remove the factor $e^{-\sum_{i=1}^n u_i(\sigma)t_i/z}$. 
Since the $\TT$-fixed point $\X(\bSigma)_\sigma$ has no Novikov variables, 
we can regard $\tilde{Q}$ in $I_\sigma^S(\tilde{Q},-z)$ 
as \emph{variables} rather than elements of the 
ground ring. (In other words, $\sL_\sigma^\tw$ is defined over 
$S_\TT$.) 
Therefore we can absorb the factor $e^{\lambda t}$ 
into $\tilde{Q}$ by rescaling $\tilde{Q}$. 
The claim follows.  

Define rational numbers $a_{ij}$ for $i\notin \sigma, j\in \sigma$ by
$\bar{\rho}_i = \sum_{j\in \sigma} a_{ij} \bar{\rho}_j$ for $1\le i\le
n$ and $\bar{s}_{i-n} = \sum_{j\in \sigma} a_{i j} \bar{\rho}_j $ for
$n+1\le i\le n+m$.  Then the equation \eqref{eq:lambda_i_relation}
shows that
\begin{equation}
\label{eq:lambdaj}
\lambda_j = - \sum_{i\notin \sigma} \lambda_i a_{ij}  
\end{equation}
for $\lambda \in \Lambda^S_\sigma$ and $j\in \sigma$.  Henceforth we
regard $\lambda_j$ for $j\in \sigma$ as a linear function of
$(\lambda_i : i\notin \sigma)$ via this relation.  We introduce
variables $(q_i: i\notin \sigma)$ dual to $(\lambda_i : i \notin
\sigma)$ and consider the change of variables:
\[
\tilde{Q}^{\lambda} = \prod_{i\notin \sigma} q_i^{\lambda_i} 
\]
We also have 
\[
v^S(\lambda) = \sum_{j \in \sigma} \lceil
\lambda_j \rceil \rho_j + \sum_{i \notin \sigma, i\le n}
\lambda_i \rho_i + \sum_{i=1}^m \lambda_{n+i}  s_i 
\equiv \sum_{i\notin \sigma}\lambda_i b^i \mod N_\sigma 
\]
where 
\[
b^i = \begin{cases} 
\text{the image of }\rho_i  \text{ in $N(\sigma) \cong \Box(\sigma)$}
& 1\le i\le n \\ 
\text{the image of }s_{i-n} \text{ in $N(\sigma) \cong \Box(\sigma)$} 
& n+1 \le i\le n+m  
\end{cases} 
\]
Now it suffices to show that
\begin{equation} 
\label{eq:restriction_of_I} 
-z \sum_{(\lambda_i : i\notin \sigma) \in (\ZZ_{\ge 0})^{\ell}}
 \left( \prod_{i\notin \sigma}
\frac{q_i^{\lambda_i}}{\lambda_i! (-z)^{\lambda_i}} \right) 
\left( 
\prod_{j \in \sigma} 
\frac{\prod_{\< a \>=\langle \lambda_j \rangle, a\leq 0}(u_j(\sigma)- a z)}
{\prod_{\<a\>=\langle \lambda_j\rangle, a\leq \lambda_j}(u_j(\sigma)- a z)}
\right) 1_{\sum_{i\notin \sigma} \lambda_i b^i} 
\end{equation} 
is a $S_\TT[\![q]\!]$-valued point on 
$\sL_\sigma^\tw$, where $\ell = n+ m -\dim \X(\bSigma)$. 

Jarvis--Kimura \cite{jk} calculated the Gromov--Witten 
theory of $BG$ with $G$ a finite group, and it follows from 
their result that the $J$-function of $BN(\sigma) \cong \X(\bSigma)_\sigma$ 
is: 
\begin{align}
\label{eq:J_BN} 
J_{BN(\sigma)} \left({\textstyle\sum_{i\notin \sigma}} q_i 1_{b^i},-z\right) 
= - z \sum_{(\lambda_i : i\notin \sigma) \in (\ZZ_{\ge 0})^{\ell }}
\left (
\prod_{i\notin \sigma }\frac{  q_i^{\lambda_i} }
{\lambda_i! (-z)^{\lambda_i}}  \right) 
1_{\sum_{i\notin \sigma} \lambda_i b^i}
\end{align} 
(See \cite[Proposition 6.1]{ccit}.)\phantom{.}  Comparing this with
\eqref{eq:restriction_of_I}, we find that the expression
\eqref{eq:restriction_of_I} is the \emph{hypergeometric modification}
of $J_{BN(\sigma)}$, in the sense of \cite{cg,ccit}.  The $J$-function
(\ref{eq:J_BN}) lies on the Lagrangian cone of the Gromov--Witten
theory of $BN(\sigma)$ (see Remark \ref{rmk:jfunction}), and we now
use the argument of \cite{ccit} to show that the hypergeometric
modification of the $J$-function \eqref{eq:restriction_of_I} lies on
the cone $\sL_\sigma^\tw$ of the twisted theory.

We briefly recall the setting from \cite{ccit}. 
Let $F$ be the direct sum $\bigoplus_{j=1}^d F^{(j)}$ of 
$d$ vector bundles and 
consider a universal multiplicative characteristic class: 
\[
\bc(F) = \prod_{j=1}^d \exp\left(\sum_{k=0}^\infty s_k^{(j)} \ch_k(F^{(j)})\right) 
\]
where $s_0^{(j)},s_1^{(j)}, s_2^{(j)}, \dots$ are formal
indeterminates.  As in Section \ref{sec:twisted}, one can define
$(F,\bc)$-twisted Gromov--Witten invariants and a Lagrangian cone for
the $(F,\bc)$-twisted theory.  The Lagrangian cone here is defined
over certain formal power series ring $\Lambda_{\rm nov}[\![\bs]\!]$
in infinitely many variables $s_k^{(j)}$, $0 \leq k < \infty$, $1\le j\le d$.  We
apply this setting to the case where $F=T_\sigma \X(\bSigma)$, which
is the direct sum of line bundles $u_j|_\sigma$, $j\in \sigma$ over
$\X(\bSigma)_\sigma$.  Denote by $\sL^\bs$ the Lagrangian cone of the
$(T_\sigma \X(\bSigma), \bc)$-twisted theory of the $\TT$-fixed point
$\X(\bSigma)_\sigma$.  By specializing the parameters $s^{(j)}_k$,
$j\in \sigma$, as
\[
s^{(j)}_k = \begin{cases} 
- \log u_j(\sigma) & k=0 \\ 
(-1)^k (k-1)! u_j(\sigma)^{-k} & k\ge 1 
\end{cases}
\]
we recover the $(T_\sigma \X(\bSigma), \eT^{-1})$-twisted theory of
$\X(\bSigma)_\sigma$. This specialization ensures that
\begin{align*}
  \bs^{(j)}(x) := \exp\left(\sum_{k=0}^\infty s_k^{(j)} \frac{x^k}{k!} 
  \right) 
  && \text{coincides with} && (u_j(\sigma)+ x )^{-1}.
\end{align*}
It now suffices to show the following: 
\begin{lem} 
Let:
\[
I_\bs(q) = \sum_{(\lambda_i: i\notin\sigma)\in (\ZZ_{\ge 0})^\ell } 
\left (
\prod_{i\notin \sigma }\frac{  q_i^{\lambda_i} }
{\lambda_i! (-z)^{\lambda_i}}  \right) 
\left( 
\prod_{j\in \sigma} 
\frac{\prod_{\langle a\rangle = \langle \lambda_j\rangle, 
a\le 0} \exp(-\bs^{(j)}(-az))}
{\prod_{\langle a\rangle = \langle \lambda_j\rangle, 
a\le \lambda_j} \exp(-\bs^{(j)}(-az))}
\right) 
1_{\sum_{i\notin \sigma} \lambda_i b^i}
\] 
where $\lambda_j$ with $j\in \sigma$ is a linear function of
$(\lambda_i : i\notin \sigma)$ via \eqref{eq:lambdaj}. Then $I_\bs(q)$
defines a $\CC[\![\bs]\!][\![q]\!]$-valued point on $\sL^\bs$.
\end{lem}
\begin{proof} 
Introduce the function:  
\[
G_y^{(j)}(x, z):=\sum_{l,m\geq 0}
s_{l+m-1}^{(j)}\frac{B_m(y)}{m!}\frac{x^l}{l!} z^{m-1}\in 
\CC[y,x, z, z^{-1}][\![s_0^{(j)}, s_1^{(j)},s_2^{(j)},\ldots]\!]
\]
as in \cite{ccit}. We have:
\begin{align}
\label{eq:property_G} 
\begin{split} 
G_y^{(j)}(x,z)& =G_0^{(j)}(x+yz, z) \\
G_0^{(j)}(x+z,z)&=G_0^{(j)}(x,z) +\bs^{(j)}(x).
\end{split} 
\end{align}
We apply the differential operator $\exp\big({-\sum_{j\in \sigma}}
G_0^{(j)}(z \theta_j, z)\big)$ with 
$\theta_j = \sum_{i\notin \sigma} a_{ij} q_i (\partial/\partial q_i)$ 
to the $J$-function \eqref{eq:J_BN} of $BN(\sigma)$ and obtain: 
\begin{multline*} 
\fb := e^{{-\sum_{j\in \sigma}}G_0^{(j)}(z \theta_j,z)} 
J_{BN(\sigma)}
\left(
{\textstyle\sum_{i\notin \sigma} q_i 1_{b^i}}, -z 
\right) \\
= 
- z \sum_{(\lambda_i : i\notin \sigma) \in (\ZZ_{\ge 0})^{\ell }}
\left (
\prod_{i\notin \sigma }\frac{  q_i^{\lambda_i} }
{\lambda_i! (-z)^{\lambda_i}}  \right) 
\exp\left (- \sum_{j\in \sigma} G_0^{(j)}(-z \lambda_j, z)\right) 
1_{\sum_{i\notin \sigma} \lambda_i b^i}
\end{multline*} 
where we used \eqref{eq:lambdaj}.  The argument in the paragraph after
\cite[equation 14]{ccit} shows that $\fb$ lies on the Lagrangian cone
$\sL^{\untw}$ of the untwisted theory of $BN(\sigma)$.  (This is where
we use Theorem~\ref{thm:lag_cone}.)\phantom{.}  On the other hand,
Tseng's quantum Riemann-Roch operator for $\bigoplus_{j\in \sigma}
u_j|_\sigma$ is:
\[
\Delta_\bs = \bigoplus_{b\in \Box(\sigma)} 
\exp\left( \sum_{j\in \sigma}G_{b_j}^{(j)}(0,z) \right) 
\]
This operator maps the untwisted cone $\sL^{\untw}$ 
to the twisted cone $\sL^\bs$ \cite{ts}. 
Therefore 
\[
\Delta_\bs \fb = 
- z \sum_{(\lambda_i : i\notin \sigma) \in (\ZZ_{\ge 0})^{\ell }}
\left (
\prod_{i\notin \sigma }\frac{  q_i^{\lambda_i} }
{\lambda_i! (-z)^{\lambda_i}}  \right) 
\exp\left (\sum_{j\in \sigma} \left(
G_{\langle -\lambda_j \rangle}^{(j)}(0,z) - G_0^{(j)}(-z \lambda_j, z) \right) \right) 
1_{\sum_{i\notin \sigma} \lambda_i b^i}
\]
lies on $\sL^\bs$. Here we used the fact that $b_j = \langle
-\lambda_j\rangle$ for the box element $b = \sum_{i\notin \sigma}
\lambda_i b^i$.  After a straightforward calculation using the
properties \eqref{eq:property_G}, the Lemma follows.
\end{proof} 

This completes the proof of Proposition~\ref{prop:restriction_for_I},
and thus completes the proof of our mirror theorem.

\bibliographystyle{amsalpha}
\bibliography{bibliography}

\providecommand{\bysame}{\leavevmode\hbox to3em{\hrulefill}\thinspace}
\providecommand{\MR}{\relax\ifhmode\unskip\space\fi MR }
\providecommand{\MRhref}[2]{%
  \href{http://www.ams.org/mathscinet-getitem?mr=#1}{#2}
}
\providecommand{\href}[2]{#2}
\begin{thebibliography}{CCFK14}

\bibitem[AGV02]{AGV1}
Dan Abramovich, Tom Graber, and Angelo Vistoli, \emph{Algebraic orbifold
  quantum products}, Orbifolds in mathematics and physics ({M}adison, {WI},
  2001), Contemp. Math., vol. 310, Amer. Math. Soc., Providence, RI, 2002,
  pp.~1--24. \MR{1950940 (2004c:14104)}

\bibitem[AGV08]{AGV2}
\bysame, \emph{Gromov-{W}itten theory of {D}eligne-{M}umford stacks}, Amer. J.
  Math. \textbf{130} (2008), no.~5, 1337--1398. \MR{2450211 (2009k:14108)}

\bibitem[BC10]{BC}
Arend Bayer and Charles Cadman, \emph{Quantum cohomology of {$[\CC^N/\mu_r]$}},
  Compos. Math. \textbf{146} (2010), no.~5, 1291--1322. \MR{2684301
  (2012d:14095)}

\bibitem[BC11]{Brini--Cavalieri}
Andrea Brini and Renzo Cavalieri, \emph{Open orbifold {G}romov-{W}itten
  invariants of {$[\Bbb C^3/\Bbb Z_n]$}: localization and mirror symmetry},
  Selecta Math. (N.S.) \textbf{17} (2011), no.~4, 879--933. \MR{2861610
  (2012i:14068)}

\bibitem[BCS05]{BCS}
Lev~A. Borisov, Linda Chen, and Gregory~G. Smith, \emph{The orbifold {C}how
  ring of toric {D}eligne-{M}umford stacks}, J. Amer. Math. Soc. \textbf{18}
  (2005), no.~1, 193--215 (electronic). \MR{2114820 (2006a:14091)}

\bibitem[BG09]{BG}
Jim Bryan and Tom Graber, \emph{The crepant resolution conjecture}, Algebraic
  geometry---{S}eattle 2005. {P}art 1, Proc. Sympos. Pure Math., vol.~80, Amer.
  Math. Soc., Providence, RI, 2009, pp.~23--42. \MR{2483931 (2009m:14083)}

\bibitem[Bro09]{brown}
Jeffrey Brown, \emph{Gromov--{W}itten invariants of toric fibrations}, To
  appear in {IMRN}. Available at
  \href{http://dx.doi.org/10.1093/imrn/rnt030}{http://dx.doi.org/10.1093/imrn/rnt030},
  2009.

\bibitem[CC09]{CC}
Charles Cadman and Renzo Cavalieri, \emph{Gerby localization, {$Z_3$}-{H}odge
  integrals and the {GW} theory of {$[\Bbb C^3/Z_3]$}}, Amer. J. Math.
  \textbf{131} (2009), no.~4, 1009--1046. \MR{2543921 (2010e:14051)}

\bibitem[CCFK14]{ccfk}
Daewoong Cheong, Ionut Ciocan-Fontanine, and Bumsig Kim, \emph{Orbifold
  quasimap theory},
  \href{http://arxiv.org/abs/1405.7160}{\texttt{arXiv:1405.7160 [math.AG]}},
  2014.

\bibitem[CCIT09]{ccit}
Tom Coates, Alessio Corti, Hiroshi Iritani, and Hsian-Hua Tseng,
  \emph{Computing genus-zero twisted {G}romov-{W}itten invariants}, Duke Math.
  J. \textbf{147} (2009), no.~3, 377--438. \MR{2510741 (2010a:14090)}

\bibitem[CCIT14]{ccit2}
\bysame, \emph{Some applications of the mirror theorem for toric stacks},
  \href{http://arxiv.org/abs/1401.2611}{\texttt{arXiv:1401.2611 [math.AG]}},
  2014.

\bibitem[CCLT13]{ChChLaTs2}
Kwokwai Chan, Cheol-Hyun Cho, Siu-Cheong Lau, and Hsian-Hua Tseng, \emph{Gross
  fibrations, {SYZ} mirror symmetry, and open {G}romov--{W}itten invariants for
  toric {C}alabi--{Y}au orbifolds},
  \href{http://arxiv.org/abs/1306.0437}{\texttt{arXiv:1306.0437 [math.SG]}},
  2013.

\bibitem[CCLT14]{ChChLaTs}
\bysame, \emph{Lagrangian {F}loer superpotentials and crepant resolutions for
  toric orbifolds}, Comm. Math. Phys. \textbf{328} (2014), no.~1, 83--130.
  \MR{3196981}

\bibitem[CFK13]{cfk}
Ionut Ciocan-Fontanine and Bumsig Kim, \emph{Wall-crossing in genus zero
  quasimap theory and mirror maps},
  \href{http://arxiv.org/abs/1304.7056}{\texttt{arXiv:1304.7056 [math.AG]}},
  2013.

\bibitem[CG07]{cg}
Tom Coates and Alexander Givental, \emph{Quantum {R}iemann-{R}och, {L}efschetz
  and {S}erre}, Ann. of Math. (2) \textbf{165} (2007), no.~1, 15--53.
  \MR{2276766 (2007k:14113)}

\bibitem[CG11]{acvg}
Alessio Corti and Vasily Golyshev, \emph{Hypergeometric equations and weighted
  projective spaces}, Sci. China Math. \textbf{54} (2011), no.~8, 1577--1590.
  \MR{2824960}

\bibitem[CIT09]{cit}
Tom Coates, Hiroshi Iritani, and Hsian-Hua Tseng, \emph{Wall-crossings in toric
  {G}romov-{W}itten theory. {I}. {C}repant examples}, Geom. Topol. \textbf{13}
  (2009), no.~5, 2675--2744. \MR{2529944 (2010i:53173)}

\bibitem[CLCT09]{cclt}
Tom Coates, Yuan-Pin Lee, Alessio Corti, and Hsian-Hua Tseng, \emph{The quantum
  orbifold cohomology of weighted projective spaces}, Acta Math. \textbf{202}
  (2009), no.~2, 139--193. \MR{2506749 (2010f:53155)}

\bibitem[CLS11]{cls}
David~A. Cox, John~B. Little, and Henry~K. Schenck, \emph{Toric varieties},
  Graduate Studies in Mathematics, vol. 124, American Mathematical Society,
  Providence, RI, 2011. \MR{2810322 (2012g:14094)}

\bibitem[Cor]{IHP}
Alessio Corti, \emph{Quantum orbifold cohomology of weak {F}ano toric {DM}
  stacks}, Talk at Institute Henri Poincar\'{e}, February 15, 2007, as part of
  the Workshop on Quantum Cohomology of Stacks and String Theory, February
  12--16, 2007.

\bibitem[CR02]{CR2}
Weimin Chen and Yongbin Ruan, \emph{Orbifold {G}romov-{W}itten theory},
  Orbifolds in mathematics and physics ({M}adison, {WI}, 2001), Contemp. Math.,
  vol. 310, Amer. Math. Soc., Providence, RI, 2002, pp.~25--85. \MR{1950941
  (2004k:53145)}

\bibitem[CR04]{CR1}
\bysame, \emph{A new cohomology theory of orbifold}, Comm. Math. Phys.
  \textbf{248} (2004), no.~1, 1--31. \MR{2104605 (2005j:57036)}

\bibitem[FLT12]{flt}
Bohan Fang, Chiu-Chu~Melissa Liu, and Hsian-Hua Tseng, \emph{Open-closed
  {G}romov-{W}itten invariants of 3-dimensional {C}alabi-{Y}au smooth toric
  {DM} stacks}, \href{http://arxiv.org/abs/1212.6073}{\texttt{arXiv:1212.6073
  [math.AG]}}, 2012.

\bibitem[FMN10]{FMN}
Barbara Fantechi, Etienne Mann, and Fabio Nironi, \emph{Smooth toric
  {D}eligne-{M}umford stacks}, J. Reine Angew. Math. \textbf{648} (2010),
  201--244. \MR{2774310 (2012b:14097)}

\bibitem[Giv95]{Givental:homological}
Alexander~B. Givental, \emph{Homological geometry. {I}. {P}rojective
  hypersurfaces}, Selecta Math. (N.S.) \textbf{1} (1995), no.~2, 325--345.
  \MR{1354600 (97c:14052)}

\bibitem[Giv98]{Givental:toric}
\bysame, \emph{A mirror theorem for toric complete intersections}, Topological
  field theory, primitive forms and related topics ({K}yoto, 1996), Progr.
  Math., vol. 160, Birkh\"auser Boston, Boston, MA, 1998, pp.~141--175.
  \MR{1653024 (2000a:14063)}

\bibitem[Giv01]{gi}
\bysame, \emph{Gromov-{W}itten invariants and quantization of quadratic
  {H}amiltonians}, Mosc. Math. J. \textbf{1} (2001), no.~4, 551--568, 645,
  Dedicated to the memory of I. G. Petrovskii on the occasion of his 100th
  anniversary. \MR{1901075 (2003j:53138)}

\bibitem[Giv04]{gi2}
\bysame, \emph{Symplectic geometry of {F}robenius structures}, Frobenius
  manifolds, Aspects Math., E36, Friedr. Vieweg, Wiesbaden, 2004, pp.~91--112.
  \MR{2115767 (2005m:53172)}

\bibitem[GS14]{gs}
Martin Guest and Hironori Sakai, \emph{Orbifold quantum {D}-modules associated
  to weighted projective spaces}, Comment. Math. Helv. \textbf{89} (2014),
  no.~2, 273--297. \MR{3225449}

\bibitem[GW12a]{gw1}
Eduardo Gonzalez and Chris~T. Woodward, \emph{Quantum cohomology and toric
  minimal model programs},
  \href{http://arxiv.org/abs/1207.3253}{\texttt{arXiv:1207.3253 [math.AG]}},
  2012.

\bibitem[GW12b]{gw2}
\bysame, \emph{A wall-crossing formula for {G}romov-{W}itten invariants under
  variation of {GIT} quotient},
  \href{http://arxiv.org/abs/1208.1727}{\texttt{arXiv:1208.1727 [math.AG]}},
  2012.

\bibitem[Iri06]{Iritani:eqf}
Hiroshi Iritani, \emph{Quantum {$D$}-modules and equivariant {F}loer theory for
  free loop spaces}, Math. Z. \textbf{252} (2006), no.~3, 577--622. \MR{2207760
  (2007e:53118)}

\bibitem[Iri09]{ir}
\bysame, \emph{An integral structure in quantum cohomology and mirror symmetry
  for toric orbifolds}, Adv. Math. \textbf{222} (2009), no.~3, 1016--1079.
  \MR{2553377 (2010j:53182)}

\bibitem[Iri11]{Iritani}
\bysame, \emph{Quantum cohomology and periods}, Ann. Inst. Fourier (Grenoble)
  \textbf{61} (2011), no.~7, 2909--2958. \MR{3112512}

\bibitem[Iwa09a]{Iwanari1}
Isamu Iwanari, \emph{The category of toric stacks}, Compos. Math. \textbf{145}
  (2009), no.~3, 718--746. \MR{2507746 (2011b:14113)}

\bibitem[Iwa09b]{Iwanari2}
\bysame, \emph{Logarithmic geometry, minimal free resolutions and toric
  algebraic stacks}, Publ. Res. Inst. Math. Sci. \textbf{45} (2009), no.~4,
  1095--1140. \MR{2597130 (2011f:14084)}

\bibitem[Jia08]{Jiang}
Yunfeng Jiang, \emph{The orbifold cohomology ring of simplicial toric stack
  bundles}, Illinois J. Math. \textbf{52} (2008), no.~2, 493--514. \MR{2524648
  (2011c:14059)}

\bibitem[JK02]{jk}
Tyler~J. Jarvis and Takashi Kimura, \emph{Orbifold quantum cohomology of the
  classifying space of a finite group}, Orbifolds in mathematics and physics
  ({M}adison, {WI}, 2001), Contemp. Math., vol. 310, Amer. Math. Soc.,
  Providence, RI, 2002, pp.~123--134. \MR{1950944 (2004a:14056)}

\bibitem[Joh14]{Johnson}
Paul Johnson, \emph{Equivariant {GW} theory of stacky curves}, Comm. Math.
  Phys. \textbf{327} (2014), no.~2, 333--386. \MR{3183403}

\bibitem[JPT11]{JPT}
P.~Johnson, R.~Pandharipande, and H.-H. Tseng, \emph{Abelian {H}urwitz-{H}odge
  integrals}, Michigan Math. J. \textbf{60} (2011), no.~1, 171--198.
  \MR{2785870 (2012c:14025)}

\bibitem[JT08]{JT}
Yunfeng Jiang and Hsian-Hua Tseng, \emph{Note on orbifold {C}how ring of
  semi-projective toric {D}eligne-{M}umford stacks}, Comm. Anal. Geom.
  \textbf{16} (2008), no.~1, 231--250. \MR{2411474 (2009h:14092)}

\bibitem[Liu13]{ccliu}
Chiu-Chu~Melissa Liu, \emph{Localization in {G}romov-{W}itten theory and
  orbifold {G}romov-{W}itten theory}, Handbook of moduli. {V}ol. {II}, Adv.
  Lect. Math. (ALM), vol.~25, Int. Press, Somerville, MA, 2013, pp.~353--425.
  \MR{3184181}

\bibitem[LS14]{sl}
Yuan-Pin Lee and Mark Shoemaker, \emph{A mirror theorem for the mirror
  quintic}, Geom. Topol. \textbf{18} (2014), no.~3, 1437--1483. \MR{3228456}

\bibitem[Man08]{Mann}
Etienne Mann, \emph{Orbifold quantum cohomology of weighted projective spaces},
  J. Algebraic Geom. \textbf{17} (2008), no.~1, 137--166. \MR{2357682
  (2008k:14106)}

\bibitem[MT08]{mt}
Todor~E. Milanov and Hsian-Hua Tseng, \emph{The spaces of {L}aurent
  polynomials, {G}romov-{W}itten theory of {$\Bbb P^1$}-orbifolds, and
  integrable hierarchies}, J. Reine Angew. Math. \textbf{622} (2008), 189--235.
  \MR{2433616 (2010e:14053)}

\bibitem[Tse10]{ts}
Hsian-Hua Tseng, \emph{Orbifold quantum {R}iemann-{R}och, {L}efschetz and
  {S}erre}, Geom. Topol. \textbf{14} (2010), no.~1, 1--81. \MR{2578300
  (2011c:14147)}

\bibitem[Vla02]{Vlassopoulos}
Yiannis Vlassopoulos, \emph{Quantum cohomology and {M}orse theory on the loop
  space of toric varieties},
  \href{http://arxiv.org/abs/math/0203083}{\texttt{arXiv:math/0203083}}, 2002.

\bibitem[Woo12]{Woodward:1}
Chris~T. Woodward, \emph{Quantum {K}irwan morphism and {G}romov-{W}itten
  invariants of quotients {I}},
  \href{http://arxiv.org/abs/1204.1765}{\texttt{arXiv:1204.1765 [math.AG]}},
  2012.

\bibitem[Woo14a]{Woodward:2}
\bysame, \emph{Quantum {K}irwan morphism and {G}romov-{W}itten invariants of
  quotients {II}},
  \href{http://arxiv.org/abs/1408.5864}{\texttt{arXiv:1408.5864 [math.AG]}},
  2014.

\bibitem[Woo14b]{Woodward:3}
\bysame, \emph{Quantum {K}irwan morphism and {G}romov-{W}itten invariants of
  quotients {III}},
  \href{http://arxiv.org/abs/1408.5869}{\texttt{arXiv:1408.5869 [math.AG]}},
  2014.

\end{thebibliography}

\end{document}